\documentclass[reqno]{amsart}

\usepackage{amssymb} 
\usepackage{amsmath}

\newtheorem{theorem}{Theorem}[section]
\newtheorem{lemma}[theorem]{Lemma}
\newtheorem{proposition}[theorem]{Proposition}
\newtheorem{corollary}[theorem]{Corollary}
\theoremstyle{definition}

\theoremstyle{remark}
\newtheorem{remark}[theorem]{Remark}
\numberwithin{equation}{section}

\newcommand{\R}{{\mathbb R}}
\newcommand{\Z}{{\mathbb Z}}
\newcommand{\G}{\frac{\alpha}{\Gamma(1-\alpha)}}

\newcommand{\e}{\epsilon}
\newcommand{\D}{\mathcal{D}}
\newcommand{\K}{\mathcal{K}}

\newcommand{\M}{\mathcal{M}}

\author{Mark Allen}
\address{Department of Mathematics, Brigham Young University, Provo,
  UT 84602}
\email{allen@mathematics.byu.edu}

\thanks{M.~Allen was supported by NSF grant DMS-1303632}

\title{H\"older regularity for nondivergence nonlocal parabolic equations}

\begin{document}

\maketitle
\makeatletter
\vspace{-2em}
{\centering\enddoc@text}
\let\enddoc@text\empty 
\makeatother

\begin{abstract}
 This paper proves H\"older continuity of viscosity solutions to certain nonlocal parabolic equations 
 that involve a generalized fractional time derivative of Marchaud or Caputo type. As a necessary and preliminary result, this paper first 
 proves H\"older continuity for viscosity solutions to certain nonlinear ordinary differential equations involving the generalized fractional time derivative.
\end{abstract}

\section{Introduction}

 This paper studies nonlocal parabolic equations of nondivergence type involving a generalized fractional time derivative.  Specifically, we study   
  \begin{equation}   \label{e:main}
   \begin{aligned}
      f(x,t) &= \D_t^{\alpha} u(x,t) - \mathcal{I} u(x,t)  \\
           &:=  \int_{-\infty}^t [u(x,t)-u(x,s)]\mathcal{K}(t,s) \ ds \\
           & \quad - \sup_{i}\inf_{j}\left( \int_{\R^n}[u(x+y,t)+ u(x-y,t)-2u(x,t)]K^{ij}(x,y,t) \ dy\right).
    \end{aligned}
  \end{equation}
 The main result is that viscosity solutions are H\"older continuous.  
 Before stating the exact assumptions made for the kernels $\K$ and $K^{ij}$,
  we first describe the recent history of showing H\"older continuity for parabolic equations of nondivergence type
 as well as the motivation for studying a parabolic equation involving a generalized fractional time derivative. 
 
  For local linear parabolic equations of nondivergence type, Krylov and Safonov \cite{ks80} proved 
 H\"older continuity of solutions without requiring any regularity assumptions on the coefficients of the equation. This method was later
 adapted to study regularity properties for fully nonlinear elliptic equations, see \cite{cc95}. Wang \cite{w92} adapted the 
 methods for nonlinear elliptic equations to prove the continuity of solutions
 to fully nonlinear parabolic equations. In recent years nonlocal equations have attracted interest due to the applications in the physical sciences
 \cite{mk00,z02}. Of particular interest are nonlocal elliptic operators of integro-differential type
  \begin{equation}  \label{e:ellnon}
   \mathcal{L} u(x) = \int_{\R^n} [u(x)-u(y)]K(x,y) \ dy.
  \end{equation}
 When $K(x,y)=c_{n,\sigma}|x-y|^{-n-2\sigma}$, then $\mathcal{L}$ is simply the fractional Laplacian $(-\Delta)^{\sigma}$. 
 An assumption such as $K(x,y)=K(x,-y)$ will make $\mathcal{L}$ a nondivergence type operator while an assumption such 
 as $K(x,y)=K(y,x)$ will give an operator of divergence type. If $\mathcal{L}$ is of nondivergence type, then $\mathcal{L}$
 is the nonlocal analogue of the linear elliptic operator 
  \[
   Lu = a_{ij} u_i u_j.
  \]
 A typical ellipticity assumption on $K(x,y)$ is that $\lambda |x-y|^{-n-2\sigma} \leq K(x,y) \leq \Lambda|x-y|^{-n-2\sigma}$, and 
 is analogous to the assumption for local equations that $\lambda |\xi|^2 \leq a_{ij}\xi_i\xi_j \leq \Lambda |\xi|^2$. Although regularity estimates
 were known for operators of type \eqref{e:ellnon}, these estimates were not uniform as the order $\sigma$ of the operator went to $1$. 
 Caffarelli and Silvestere \cite{cs09}  adapted the techniques for fully nonlinear local equations \cite{cc95} to the nonlocal setting, and obtained uniform 
 estimates as the order of the operator $\sigma$ went to $1$. Chang Lara and Davila  \cite{cl141} then proved regularity estimates for parabolic 
 equations of the type
  \[
   \partial_t u(x,t) - \mathcal{I}u(x,t) = 0,
  \]
 where $\mathcal{I}$ is a nonlocal nonlinear elliptic operator. The estimates in \cite{cl141} are uniform as the order of the operator $\sigma$ approaches 
 $1$, so that the results in \cite{cl141} are not only a nonlocal analogue, 
 but also recover many of the regularity results for local parabolic equations. Recently, Schwab and Silvestre in \cite{ss16}
 have proven H\"older regularity for parabolic equations with
 even more general nonlocal elliptic terms.

 The equation of interest in this paper is motivated by the equation
  \begin{equation}   \label{e:plasma}
   \partial_t^{\alpha} u - D_{|x|}^{\beta} u = f,  
  \end{equation}
 introduced in \cite{dcl04,dcl05} to model plasma transport. The function $u$ is the probability density function for tracer particles in the plasma which 
 represents the probability of finding a particle at time $t$ and position $x$. The right hand side $f$ is a source term. 
 The nonlocal diffusion operator $D_{|x|}^{\beta}$ is one-dimensional and 
 accounts for avalanche-like transport. The fractional derivative $\partial_t^{\alpha}$ is the Caputo derivative, and in the context of this model the Caputo 
 derivative accounts for the trapping of the trace particles in turbulent eddies. Although \eqref{e:plasma} is an equation for one spatial dimension, we 
 may consider the following nonlocal parabolic equation in higher dimensions
  \begin{equation}    \label{e:firstdive}
   \partial_t^{\alpha}u  - \mathcal{L}u = f
  \end{equation}  
 where $\mathcal{L}$ is nonlocal elliptic operator of type \eqref{e:ellnon}, and $\partial_t^{\alpha}$ is the Caputo derivative. 
 As mentioned earlier, certain assumptions made on the kernel $K(x,y)$ lead to an operator of either divergence or nondivergence type. 
 The author in \cite{a17} studied weak solutions of \eqref{e:firstdive} of nondivergence type. 
 The assumptions made on the kernel $K(x,y,t):\R^n \times \R^n \times \R \to \R$ were  that 
  \begin{equation}   \label{e:elliplamb}
   \frac{\lambda}{|x-y|^{n+\sigma}}\leq   K(x,y,t) \leq \frac{\Lambda}{|x-y|^{n+\sigma}},  
  \end{equation} 
  for fixed $0<\lambda \leq \Lambda$, and that $K(x,y,t)=K(x,-y,t)$. 
 For divergence form equations, the author in 
 \cite{acv16} studied weak solutions of \eqref{e:firstdive} of divergence type and proved a De Giorgi-Nash-Moser type theorem which gives 
 H\"older continuity of solutions.  The results in \cite{acv16} assumed \eqref{e:elliplamb} and $K(x,y,t)=K(y,x,t)$. 
 The latter assumption leads to \eqref{e:firstdive} being an equation of divergence type. It was shown in 
 \cite{acv17} that the methods utilized in 
 \cite{acv16} can be applied to an equation of type \eqref{e:firstdive} with a much more general fractional time derivative. 
 The main aim of this paper is to extend the results for nondivergence form equations 
 in \cite{a17} to equations involving this more general fractional time derivative
  which we now describe.

 \subsection{The Marchaud and Caputo Derivatives} 
  The Caputo derivative is useful in modeling phenomena that take into account interactions in the past. In this manner one can think of 
  the equation has having ``memory''. This is in contrast to parabolic equations such as the heat operator $\partial_t - \Delta$
  that do not account for the past. The Caputo derivative 
  is defined as 
   \[
    _a\partial_t^{\alpha} f(t) := \frac{1}{\Gamma(1-\alpha)} \int_{a}^t \frac{f'(s)}{(t-s)^{\alpha}} \ ds.
   \]
  For $C^1$ functions one may use integration by parts to show the equivalent formula
   \[
    _a\partial_t^{\alpha} f(t) = \frac{1}{\Gamma(1-\alpha)} \frac{f(t)-f(a)}{(t-a)^{\alpha}} + \G \int_a^t \frac{u(t)-u(s)}{(t-s)^{1+\alpha}} \ ds.
   \]
  If we define $f(t)=f(a)$ for $t<a$ then as in \cite{acv16} we have the equivalent formulation 
   \begin{equation}  \label{e:fracform}
    \partial_t^{\alpha} f(t) := \G \int_{-\infty}^t \frac{f(t)-f(s)}{(t-s)^{1+\alpha}} \ ds.
   \end{equation}
  This one-sided nonlocal derivative is known as the Marchaud derivative \cite{skm93}, and was recently studied 
   in \cite{bmst16}. The formulation in \eqref{e:fracform} is very useful. It is no longer essential to know, and therefore label, the initial point $a$. 
  Another useful feature of this formulation in \eqref{e:fracform}
  is rather than  assigning initial data as simply $u(a)=c$,  one may   
  assign more general ``initial'' data  as $u(t)=\phi(t)$ for $t\leq a$ with $\phi(x)$  not necessarily differentiable or even continuous. 
  The formulation in \eqref{e:fracform} will be particularly useful for the notion of 
  viscosity solutions in the context of nondivergence solutions later described in Section \ref{s:viscosity}. 
  The formulation in \eqref{e:fracform} looks similar to the one-dimensional fractional Laplacian 
  except that the integration occurs from only one side. We may then apply many of the techniques developed for nonlocal elliptic 
  operators like the fractional Laplacian to equations involving the Caputo derivative as was done in \cite{acv16,acv17}. 
  Finally, this formulation in \eqref{e:fracform} allows for a different type of generalization of the Caputo derivative. Rather than generalizing as
   \[
    \frac{d}{dt} (k \ast (f-f(a))) \quad \text{ or } \quad \int_a^t f'(s)K(t-s) \ ds,
   \]
   one may generalize as 
   \begin{equation}  \label{e:maintime}
    \D_t^{\alpha}f(t):= \int_{-\infty}^t [f(t)-f(s)]\K(t,s) \ ds. 
   \end{equation}
  The proof of H\"older continuity in \cite{acv16} for the linear divergence equation works for the more general fractional time derivative 
   \[
    \int_{-\infty}^t [f(t)-f(s)]\K(t,s,x) \ ds,
   \]
  provided that the kernel $\K(t,s,x): \R \times \R \times \R^n \to \R$ satisfies 
   \begin{equation}  \label{e:timediv}
    \K(t,t-s,x) = \K(t+s,t,x)
   \end{equation} 
     and
   \begin{equation}  \label{e:alphabound}
    \frac{\alpha}{\Gamma(1-\alpha)}\frac{\lambda}{(t-s)^{1+\alpha}} 
     \leq \K(t,s,x)
     \leq \frac{\alpha}{\Gamma(1-\alpha)} \frac{\Lambda}{(t-s)^{1+\alpha}}.
   \end{equation}
  See for instance \cite{acv17} where a kernel $\K(t,s,x)$ satisfying only \eqref{e:timediv} and \eqref{e:alphabound} is utilized. 
   The assumption \eqref{e:alphabound} is analogous to the assumption 
  \eqref{e:elliplamb} for the kernel $K(x,y,t)$ of $\mathcal{L}$. 
  Likewise, the condition \eqref{e:timediv} is analogous to $K(x,y,t)=K(y,x,t)$ and was 
  necessary in \cite{acv17}
  because the equation was of divergence form. In this paper we will assume that the time kernel $\K(t,s): \R \times \R \to \R$ is independent of the 
  spatial variable $x$. 

 \subsection{Main Results}
   
  We will assume that the kernels $K^{ij}$ satisfy \eqref{e:elliplamb} and $K^{ij}(x,y,t)=K^{ij}(x,-y,t)$.  
  We will also assume that the kernel $\K(t,s)$ satisfies \eqref{e:alphabound}.

  In order to prove H\"older continuity of solutions to nondivergence parabolic equations involving the Caputo derivative, 
  the author in \cite{a17} followed the idea in \cite{s11} to solve an ordinary differential equation in order to capture information backwards in 
  time for a solution. 
   In the context of solutions to \eqref{e:main} this requires solving 
   \begin{equation}    \label{e:odeode}
    \begin{cases}
     m(t)= 0 &\text{ if } t\leq -2   \\
     \D_t^{\alpha} m(t)= c_0|\{x \in B_1 : u(x,t)\leq 0\}| - C_1m(t) &\text{ if } t \in (-2,0).
    \end{cases}
   \end{equation}
  Several complications arise when considering solutions of \eqref{e:odeode}. First, we need to prove existence of solutions. Second, the right hand side is not 
  necessarily continuous. Third, a solution of \eqref{e:odeode} may not be regular enough to utilize as a part of a test function for the notion 
  of viscosity solution. Fourth, we need to show roughly that if $|\{x \in B_1\times (-2,-1): u(x,t)\leq 0\}|\geq \mu_1$, then 
  $m(t)\geq \mu_2$ if $t \in (-1,0)$. 
  When $\D_t^{\alpha}=\partial_t^{\alpha}$ the author in \cite{a17} utilized explicit representation formulas \cite{d04} for solutions to \eqref{e:odeode} to 
   obtain the necessary 
  properties.
  Since no such formulas are available for solutions to \eqref{e:odeode} in general, we overcome the four mentioned complications in a 
  different way.  We first show a priori H\"older continuity estimates for such ordinary differential equations. Our class of weak solutions will 
  be considered in the viscosity sense as described in Section \ref{s:viscosity}. Our first main result is

 \begin{theorem}  \label{t:main}
   Let $f$ be a continuous function on $[-1,0]$, and let $u$ be a bounded and a viscosity solution to 
    \[
     \sup_{k} \inf_{l} \left(\int_{-\infty}^t [u(x,t)-u(x,s)]\K^{kl}(t,s) \ ds  \right) = f(t)
    \]
   on  $(-1,0]$ 
   with $f \in L^{\infty}$. Assume also that the kernels $\K^{kl}(t,s)$ satisfy \eqref{e:alphabound}. 
   Then  if
   $0< \alpha \leq 1$ there exists two constants $0<\beta\leq 1$ and $C>0$ depending on $ \alpha,\lambda,\Lambda$ and 
   $\displaystyle \| u \|_{L^{\infty}((-\infty,-1))}$
   but uniform as $\alpha \to 1$ such that
    \begin{equation}   \label{e:fholder}
     \| u \|_{C^{0,\beta}([-1/2,0])} \leq C \| f \|_{L^{\infty}}. 
    \end{equation}
  \end{theorem}
  
 We remark that the operator in Theorem \ref{t:main} appears at first glance to be a nonlocal elliptic operator in one dimension, and
 results regarding H\"older regularity for general nonlocal elliptic operators are contained in both \cite{krs14} and \cite{ss16}. 
 However, since the integration only occurs from the left, Theorem \ref{t:main} is not contained in the results of \cite{krs14} and \cite{ss16}. 
 Specifically, $\K(t,s)$ does not belong to $A_{\text{sec}}$ in \cite{krs14}, and assumption $(A3)$ is violated in \cite{ss16}. 
 
 We utilize Theorem \ref{t:main} to prove existence of solutions to \eqref{e:odeode}. To accommodate the second and third complications
 we show in Section \ref{s:barrier} that we may approximate \eqref{e:odeode} uniformly from below by Lipschitz subsolutions to \eqref{e:odeode}
 which will be sufficient for the purposes of this paper. The fourth complication, however, remains. 
 The H\"older estimate in \eqref{e:fholder} depending on an $L^{\infty}$ norm of $f$ is insufficient to overcome the fourth complication. 
 We therefore consider a subclass of kernels $\K(t,s)$ which we assume 
 also satisfy \eqref{e:timediv}. This last assumption allows us to prove our second main result. 
 
  \begin{theorem}   \label{t:main2}
  Let $u$ be a bounded, continuous viscosity solution to \eqref{e:main} in $B_2 \times [-2,0]$.  
  Assume the kernels $K^{ij}(x,y,t)$ satisfy $K^{ij}(x,y,t)=K^{ij}(x,-y,t)$ and \eqref{e:elliplamb}. Assume the kernel $\K(t,s)$
  satisfies \eqref{e:timediv} and \eqref{e:alphabound}. 
  Then there exists three constants   $C,\kappa,\e_0>0$ depending only on $\Lambda,\lambda, n,\alpha,\sigma$, but uniform as $\alpha \to 1$
  such that $u$ is H\"older continuous in $B_1 \times [-1,0]$, and for $(x,t),(y,s) \in B_1 \times [-1,0]$ the following estimates holds 
   \[
    |u(x,t)-u(y,s)| \leq C (\| u\|_{L^{\infty}} + \epsilon_0^{-1} \| f\|_{L^{\infty}})|x-y|^{\kappa}+|t-s|^{\kappa \alpha/(2\sigma)}.
   \] 
 \end{theorem}
  
  We note that the estimates in Theorem \ref{t:main2} remain uniform as the order $\alpha$ of 
  the fractional time derivative approaches $1$. However, the estimates do not remain uniform as the order of the elliptic operator $\sigma$
  approaches $1$. Having obtained the necessary theory for solutions to ordinary differential equations involving $\D_t^{\alpha}$,
   the author plans in a future work to utilize the techniques of Chang Lara and Davila in \cite{cl141} to prove Theorem
   \ref{t:main2} with estimates uniform as $\sigma$ approaches $1$. 
   What would also be of interest is to prove 
  H\"older regularity to solutions of an equation of the form
      \begin{equation}  \label{e:nontime}
    \begin{aligned}
     &\sup_{k}\inf_{l}\left( \int_{-\infty}^t [u(x,t)-u(x,s)]\K^{kl}(t,s,x) \ ds \right) \\
      & \quad - \sup_{i}\inf_{j}\left( \int_{\R^n}u(x+y,t)-u(x,t)K^{ij}(x,y,t)  \ dy \right) = f(x,t).
    \end{aligned}
   \end{equation}
  To do so it appears necessary to prove Theorem \ref{t:main} with the H\"older estimate depending on the $L^p$ norm of the right hand side $f$. 
   
 \subsection{Notation}
  We here define notation that will be consistent thoughout the paper. 
   \begin{itemize}
    \item $\partial_t^{\alpha}$ - the Caputo derivative as defined in \eqref{e:fracform}.
    \item $\D_t^{\alpha}$ - the generalized Marchaud derivative as defined in \eqref{e:maintime}
        \item $\alpha$ - will always denote the order of the fractional time  derivative.
    \item $\sigma$ - will always denote the order of the nonlocal elliptic spatial operator. 
    \item $t,s$ - will always be variables reserved as time variables. 
    \item $K(x,y,t)$ - the kernel for the elliptic operator $\mathcal{L}$ as defined in \eqref{e:ellnon}.
    \item $\K(t,s)$ - the kernel for $\D_t^{\alpha}$ as defined in \eqref{e:maintime}.
    \item $M_{\sigma}^{\pm}$ - Pucci's extremal operators (defined in Section \ref{s:viscosity}) for the elliptic spatial operators. 
    \item $M_{\alpha}^{\pm}$ - Pucci's extremal operators (defined in Section \ref{s:viscosity}) for the fractional time derivatives.
    \item $\lambda, \Lambda$ - Ellipticity constants as appearing in \eqref{e:elliplamb} and \eqref{e:alphabound}. 
    \item $Q_r(x_0,t_0)$ - the space-time cylinder $B_r(x_0) \times (t_0 - r^{2\sigma/\alpha} , t_0)$.
    \item $Q_r$ - the cylinder centered at the origin $Q_r(0,0)$. 
   \end{itemize}
   
 \subsection{Outline}
  The outline of our paper is as follows:
  In Section \ref{s:viscosity} we explain the notion of viscosity solution that will be used in the paper. 
  In Section \ref{s:unique} we use a standard method to show the comparison principle
  and uniqueness for ordinary differential equations involving $\D_t^{\alpha}$.  
  In Section \ref{s:holder} we prove our first main result that solutions to certain ordinary differential equations 
  involving  $\D_t^{\alpha}$ are H\"older continuous. 
  In Section \ref{s:barrier} we establish the necessary properties for solutions of \eqref{e:odeode}. 
  In Section \ref{s:holder2} we prove our second main result that solutions to parabolic equations of type \eqref{e:main}
  are H\"older continuous.

\section{Viscosity Solutions and Pucci's Extremal Operators}  \label{s:viscosity}
 In order to study weak solutions of equations of type $\D_t^{\alpha} u - \mathcal{L} u=f$ in nondivergence form we will utilize the notion of viscosity solution. 
 In order to consider $u$ as a solution on $Q_r(x,t)$, we must have $u: (\R^n \times (t-r^{2\sigma/\alpha},t)) \cup (B_r(x) \times (-\infty,t)) \to \R$. The right hand side 
 $f: Q_r(x,t) \to \R$. 
 For the elliptic operator $\mathcal{L}$ we recall the notion of Pucci's extremal operators introduced in \cite{cs09}.
 For fixed time $t$ we denote the second order
 difference $\delta(u,x,y,t):= u(x+y,t)+u(x-y,t)-2u(x,t)$. We fix  two constants $0 < \lambda \leq \Lambda$ and define
  \[
   \begin{aligned}
    M_\sigma^+ u(x,t) &:= \int_{\R^n} \frac{\Lambda \delta(u,x,y,t)_+ - \lambda\delta(u,x,y,t)_-}{|y|^{n+2\sigma}} \\
    M_\sigma^- u(x,t) &:= \int_{\R^n} \frac{\lambda \delta(u,x,y,t)_+ - \Lambda\delta(u,x,y,t)_-}{|y|^{n+2\sigma}}.
   \end{aligned}
  \]
 We now define a Pucci-type extremal operator for fractional derivatives of type \eqref{e:maintime}. For fixed $x \in \R^n$,
  \[
   \begin{aligned}
    M_{\alpha}^+ u(x,t) &:= \frac{\alpha}{\Gamma(1-\alpha)} \int_{-\infty}^t \frac{\Lambda (u(x,t)-u(x,s))_+ - \lambda (u(x,t)-u(x,s))_- }{(t-s)^{1+\alpha}} \\
    M_{\alpha}^- u(x,t) &:= \frac{\alpha}{\Gamma(1-\alpha)} \int_{-\infty}^t \frac{\lambda (u(x,t)-u(x,s))_+ - \Lambda (u(x,t)-u(x,s))_- }{(t-s)^{1+\alpha}}.
   \end{aligned}
  \] 
  Since $\alpha$ is reserved for $\D_t^{\alpha}$ and $\sigma$ is reserved for the kernel of $\mathcal{L}$, there should be no confusion 
  between $M_{\alpha}^{\pm}$ and $M_{\sigma}^{\pm}$. 
  These operators give rise to the equations
  \[
  \begin{aligned}  
   &M_{\alpha}^- u(x,t) - M^+ u(x,t) \leq f(x,t)    \\
   &M_{\alpha}^+ u(x,t) - M^- u(x,t) \geq f(x,t).    
  \end{aligned}
  \]
  Since in this paper we show regularity for the parabolic equation when the kernels $\K(t,s)$ satisfy \eqref{e:timediv}, we will only consider solutions to 
  \begin{alignat}{2}  
   &\D_t^{\alpha} u(x,t) - M^+ u(x,t) \leq f(x,t)   \label{e:p3} \\
   &\D_t^{\alpha} u(x,t) - M^- u(x,t) \geq f(x,t).   \label{e:p4} 
  \end{alignat} 
  When proving regularity for ordinary differential equations we will not assume 
  $\K(t,s)$ satisfies \eqref{e:timediv}, so we consider solutions to 
  \begin{alignat}{2}  
   &M_{\alpha}^- u(t) \leq f(t)   \label{e:p5} \\
   &M_{\alpha}^+ u(t)  \geq f(t).   \label{e:p6} 
  \end{alignat}   
  
  As in \cite{cc95} we have the following properties for Pucci's extremal operators. 
  \begin{proposition}  \label{p:basics}
   For fixed $u,v$ evaluated at fixed $(x,t)$ we have the following properties where $M^{\pm}$ 
    denote either $M_{\alpha}^{\pm}$ or $M_{\sigma}^{\pm}$. 
    \[
     \begin{aligned}
       (i) \quad  &M^- u \leq M^+ u. \\
       (ii) \quad &M^- u = -M^+ (-u). \\
       (iii) \quad &M^{\pm} cu = c M^{\pm} u \text{ if } c\geq0. \\
       (iv) \quad &M^+ u   + M^- v \leq M^+(u+v) \leq M^+ u + M^+ v.\\
        (v) \quad &M^- u + M^- v \leq M^- (u+v) \leq M^-u + M^+ v.   
     \end{aligned}
    \]
  \end{proposition}
  
  We now define a viscosity solution.  
 We say that an upper semi-continuous function $u$ is a viscosity subsolution of \eqref{e:main} (or a solution of \eqref{e:p3}) in $Q_r$ if whenever 
 a $C^{2,1}$ function satisfies $\phi \geq u$ on $[t_0-r,t_0] \times B_{\rho}(x_0)$ and $\phi(x_0,t_0)=u(x_0,t_0)$ with $(x_0,t_0) \in Q_r$, and if $v$ is defined as 
  \[ 
   v(x,t) := 
    \begin{cases}
      \phi(x,t) & \text{if } (x,t) \in[t_0-r,t_0]\times B_{\rho}(x_0) \\ 
      u(x,t) & \text{ otherwise } ,\\
    \end{cases}
  \]
 Then $\D_t^{\alpha} v(x_0,t_0) - \mathcal{I}v(x_0,t_0) \leq f(x_0,t_0)$ (or $v$ is a solution to \eqref{e:p3} at $(x_0,t_0)$). 
  A viscosity supersolution of \eqref{e:main} (or a solution of \eqref{e:p4}) for lower semi-continuous functions
 is defined similarly. We point out that a viscosity subsolution (supersolution) of \eqref{e:main} is a viscosity solution of \eqref{e:p3} (\eqref{e:p4}). 
 A solution is both a subsolution and supersolution, and consequently a continuous function. 
 
 The notion of viscosity solutions and supersolutions for $\D_t^{\alpha} u = f$ (or solutions to \eqref{e:p5} and \eqref{e:p6}) are similarly defined. 
 We note that we may extend our class of test functions that touch from above or below to functions $\phi$ that are $C^2$ in the $x$-variable for 
 fixed $t$ and Lipschitz in time for fixed $x$. It is clear that if function $\D_t^{\alpha} u$ can be evaluated classically and solves $\D_t^{\alpha} u=f$ then $u$ is  
 a solution in the viscosity sense. This is made clear in the following two Propositions.
%
 
 \begin{proposition}   \label{p:classic}
  Let $u$ be a continuous bounded function. Let $\phi \in C^{0,\gamma}$ with $\alpha < \gamma \leq 1$. If $\phi \geq (\leq )u$ on $[t_0 - \e, t_0]$ and 
  $\phi(t_0)=u(t_0)$, then the integral
   \[
    \int_{-\infty}^{t_0} [u(t_0)-u(s)] \K(t,s)\ ds 
   \]
   is well defined and possibly $\infty$   $\ ( -\infty  )$ so that $\D_t^{\alpha} u (t_0)$ is well defined. 
 \end{proposition}

 \begin{proof}
  Without loss of generality we assume that $\phi \geq u$ on $[t_0 - \e, t_0]$. Then 
   \[
    \begin{aligned}
    &\int_{-\infty}^{t_0} [u(t_0) - u(s)] \K(t,s) \ ds \\
    &\quad =\int_{-\infty}^{t_0 - \e} [u(t_0) - u(s)] \K(t,s) \ ds + \int_{t_0 - \e}^{t_0} [u(t_0) - u(s)] \K(t,s) \ ds\\
       &\quad\geq \int_{-\infty}^{t_0 - \e} [u(t_0) - u(s)] \K(t,s) \ ds + \int_{t_0 - \e}^{t_0} [u(t_0) - \phi(s)] \K(t,s) \ ds\\
       &\quad= \int_{-\infty}^{t_0 - \e} [u(t_0) - u(s)] \K(t,s) \ ds + \int_{t_0 - \e}^{t_0} [\phi(t_0) - \phi(s)] \K(t,s) \ ds\\ 
       &\quad\geq \int_{-\infty}^{t_0 - \e} [u(t_0) - u(s)] \K(t,s) \ ds - \Lambda \| \phi \|_{C^{0,\gamma}} \int_{t_0 - \e}^{t_0} (t_0 - s)^{\gamma-1-\alpha} \ ds\\ 
       &\quad\geq \int_{-\infty}^{t_0 - \e} [u(t_0) - u(s)] \K(t,s) \ ds - C_1 \\
       &\quad \geq -2\Lambda \| u \|_{L^{\infty}} \int_{-\infty}^{t_0 - \e} (t_0-s)^{-1-\alpha} \ ds - C_1 \\
       &\quad \geq -C_2. 
    \end{aligned}
   \]
  Therefore, the integral is well defined and possibly $\infty$. 
 \end{proof}
 
 \begin{proposition}   \label{p:classic2}
  Let $u$ be a continuous bounded function on $(-\infty,T]$ and assume that for some $t \in (-\infty,T]$ there exists a Lipschitz function
  touching $u$ by below (above) at $t$. From Proposition \ref{p:classic}, the term $\D_t^{\alpha} u(t)$ is well defined and we have that
    \begin{equation}   \label{e:pclassic}
     \int_{-\infty}^{t} [u(t)-u(s)] \K(t,s)\ ds \geq (\leq) f(t)
    \end{equation}   
   if and only if $\D_t^{\alpha} u(t)\leq (\geq) f(t)$ in the viscosity sense. 
 \end{proposition}
 
 \begin{proof}
  Assume the inequality in \eqref{e:pclassic}. If $\phi$ touches $u$ from below in $[t-\e,t]$, then 
   \[
    \begin{aligned}
    f(t)&\leq \int_{-\infty}^{t} [u(t) - u(s)] \K(t,s) \ ds \\
    &\quad =\int_{-\infty}^{t - \e} [u(t) - u(s)] \K(t,s) \ ds + \int_{t - \e}^{t} [u(t) - u(s)] \K(t,s) \ ds\\
       &\quad\leq \int_{-\infty}^{t - \e} [u(t) - u(s)] \K(t,s) \ ds + \int_{t - \e}^{t} [u(t) - \phi(s)] \K(t,s) \ ds\\
       &\quad= \int_{-\infty}^{t - \e} [u(t) - u(s)] \K(t,s) \ ds + \int_{t - \e}^{t} [\phi(t) - \phi(s)] \K(t,s) \ ds.
    \end{aligned}
   \]
   and so $\D_t^{\alpha} u \geq f(t)$ in the viscosity sense.  Assume now that $\D_t^{\alpha} u \geq f(t)$ in the viscosity sense. 
   From the assumption, we may touch $u$ from below by a Lipschitz function $\phi$ in some neighborhood $[t-\e,t]$. 
   We may then find Lipschitz $\phi_k$ converging uniformly to $u$ in $[t - \e,t]$ with $\phi_k \geq u$ in $[t-\e,t]$. 
   Since the integral in \eqref{e:pclassic} is well defined we have from Lebesgue's dominated convergence theorem
    \[
     \begin{aligned}
     f(t) &\leq \lim_{k \to \infty}\int_{-\infty}^{t - \e} [u(t) - u(s)] \K(t,s) \ ds + \int_{t - \e}^{t} [\phi_k(t) - \phi_k(s)] \K(t,s) \ ds \\
          &= \lim_{k \to \infty} \int_{-\infty}^{t - \e} [u(t) - u(s)] \K(t,s) \ ds + \int_{t - \e}^{t} [u(t) - \phi_k(s)] \K(t,s) \ ds \\
          &= \int_{-\infty}^{t - \e} [u(t) - u(s)] \K(t,s) \ ds + \int_{t - \e}^{t} [u(t) - u(s)] \K(t,s) \ ds \\  
          &=  \int_{-\infty}^{t} [u(t) - u(s)] \K(t,s) \ ds.
      \end{aligned}
      \] 
  \end{proof}

  The notion of continuity is important for viscosity solutions. If we let 
  \[
   \begin{cases}
    u(t) = t^{\alpha - 1} &\text{ if } t>0 \\
    u(t) =0 &\text{ if }  t\leq 0,
   \end{cases}
  \]
  then one may explicitly compute that $\partial_t^{\alpha} u(t)=0$ for any $t \in \R$.  However, $u$ is not upper semi-continuous, and 
  therefore not a viscosity subsolution. 
   
 Viscosity solutions are closed under appropriate limits. 
  \begin{lemma}  \label{l:limit}
   Let $u_k$ be a sequence of continuous bounded viscosity solutions to \eqref{e:p3} (or \eqref{e:p4}) in $B_{R_1} \times [-R_2,0]$, 
   iconverging in $\R^{n} \times (-\infty,T]$ to $u_0$ bounded and continuous. Then 
   $u_0$ is a viscosity solution to \eqref{e:p3} (or \eqref{e:p4}) in $B_{R_1} \times [-R_2,0]$.
  \end{lemma}
  
  \begin{proof}
   The proof is standard and straightforward from the definition of viscosity solutions (see \cite{cc95}). 
  \end{proof}

  \section{Approximating Solutions} \label{s:unique}
   In order to show uniqueness for an ordinary differential equation such as $\D_t^{\alpha} u =f$ we use the notion of sup- and inf-convolution. 
   For a bounded and upper-semicontinuous function $u$ on $(-\infty,t_0]$, for $t\leq t_0$ we define
    \begin{equation}   \label{e:jensenup}
     u^{\e}(t):= \sup_{s\leq t} \{u(s) +\frac{1}{\e}(s-t) + \e\}. 
    \end{equation}
   If $u$ is bounded and lower-semicontinuous on $(-\infty,t_0]$,  for $t\leq t_0$ we define 
    \begin{equation}   \label{e:jensendown}
     u_{\e}(t):= \inf_{s\leq t} \{u(s) -\frac{1}{\e}(s-t) - \e\}.    
    \end{equation}
  
  We have the following properties 
   \begin{proposition}   \label{p:jensen1}
    For $u^{\e}$ as defined in \eqref{e:jensenup} and $t\leq t_0$ the following hold 
      \begin{alignat*}{2}
       &(1) \quad \text{there exists } t^* \leq t \text{ such that } u^{\e}(t)=u(t^*) + \frac{1}{\e}(t^*-t) + \e. \\
       &(2) \quad u^{\e}(t)\geq u(t)+ \e. \\
       &(3) \quad u^{\e}(t_2)-u^{\e}(t_1) \geq \e^{-1}(t_2-t_1) \text{ for } t_1 < t_2. \\
       &(4) \quad 0<\e_1<\e_2  \Rightarrow u^{\e_1}(t) \leq u^{\e_2}(t). \\
       &(5) \quad t-t^* \leq 2\e \sup |u|. \\
       &(6) \quad 0<u^{\e}(t)-u(t) \leq u(t^*) - u(t) + \e. 
      \end{alignat*}
   \end{proposition}
  
  \begin{proof}
   All properties except for $(3)$ are as in Lemma 5.2 in \cite{cc95}. For property $(3)$ we note that for $t\leq t_1 \leq t_2$ we have
    \[
     \begin{aligned}
      u^{\e}(t_2) &\geq u(t) + \frac{1}{\e}(t-t_2) + \e \\
                        &= u(t) +\frac{1}{\e}(t-t_1) + \e + \frac{1}{\e}(t_1 - t_2).
     \end{aligned}
    \]
   Taking the supremum over $t\leq t_1$ we obtain 
    \[
     u^{\e}(t_2) - u^{\e}(t_1) \geq -\e^{-1}(t_2 - t_1).
    \]
  \end{proof}
  
  Using the properties listed in Proposition \ref{p:jensen1}, it is standard to show the following Proposition which 
  is analogous to Theorem 5.1 in \cite{cc95} and Propositions 5.4 and 5.5 in \cite{cs09}. 
  \begin{proposition}  \label{p:jensen2}
   If $u$ is bounded and lower-semicontinuous (upper-semicontinuous) in $(-\infty,T]$, then $u_{\e} (-u^{\e}) \ \Gamma$-converges to
   $u (-u)$. If $u$ is continuous then $u^{\e},u_{\e}$ converge uniformly to $u$. Furthermore, if $\D_t^{\alpha} u \geq (\leq)f$, then 
   there exists $d_{\e} \to 0$ as $\e \to 0$ such that $\D_t^{\alpha} u_{\e} \geq f- d_{\e} \ (\D_t^{\alpha}u^{\e}  \leq f + d_{\e})$. 
  \end{proposition}
  
  \begin{lemma}  \label{l:jensen}
   Let $u$ be bounded and upper-semicontinuous and $v$ be bounded and lower semi-continuous on $(-\infty,T]$. Let $f,g$ be continuous
   functions. If 
   $\D_t^{\alpha} u \leq f$ and $\D_t^{\alpha} v \geq g$ in the viscosity sense, then $\D_t^{\alpha} (v-u) \geq g-f$ in the viscosity sense. 
  \end{lemma}
  
  \begin{proof}
   We take the approximating solutions $u^{\e},v_{\e}$ with $\D_t^{\alpha} u^{\e}\leq f + d_{\e}$ and $\D_t^{\alpha} v_{\e} \geq f - d_{e}$. 
   From property $(3)$ in Proposition \ref{p:jensen1} we have that at every point $u^{\e}$ can be touched from above by a Lipschitz function
   and $v_{\e}$ can be touched from below. Then also $v_{\e}-u^{\e}$ can be touched from below by a Lipschitz function at every point. 
   From Propositions \ref{p:classic} and \ref{p:classic2}, the terms $\D_t^{\alpha} u^{\e}, \D_t^{\alpha} v_{\e}, D_t^{\alpha} (v_{\e}-u^{\e})$ are well defined and 
   at any point $t$
    \[
     \D_{t}^{\alpha} (v_{\e} - u^{\e}(t)) = \D_t^{\alpha} v_{\e}(t) - \D_t^{\alpha} u^{\e}(t) \geq g -f - 2d_{\e}. 
    \]
   Then from Proposition \ref{p:classic2} we conclude that $\D_{t}^{\alpha} (v_{\e} - u^{\e}) \geq g-f-2d_{\e}$. Letting $\e \to 0$
   we obtain from Lemma \ref{l:limit} that $\D_{t}^{\alpha} (v - u) \geq g-f$ in the viscosity sense. 
  \end{proof}
  
  We are now able to prove a comparison principle. 
  \begin{theorem}  \label{t:comparison}
   Let $u$ be bounded and upper-semicontinuous and $v$ be bounded and lower semi-continuous on $(-\infty,T_2]$.  
   Let $f$ be a continuous function and assume that $\D_t^{\alpha} u \leq f \leq \D_t^{\alpha} v$ on $(T_1,T_2]$ with 
   $u \leq v$ on $(-\infty,T_1]$. Then $u\leq v$ on $(-\infty,T_2]$, and if $u(t_0)=v(t_0)$ for some $t_0 \in (T_1,T_2]$, 
   then $u(t)=v(t)$ for all $t \leq t_0$. 
  \end{theorem}
  
  \begin{proof}
   Suppose there exists $t_1 \in (T_1,T_2]$ such that $v(t_1)-u(t_1)\leq 0$. Let $v-u$ achieve its minimum in $[T_1,T_2]$ at $t_0 \in (T_1,T_2]$. 
   From Lemma \ref{l:jensen}, we have $\D_{t}^{\alpha} (v-u)\geq 0$ in the viscosity sense. Since $v-u \geq 0$
   for $t\leq T_1$, then $(v-u)$ is touched from below by the constant function $v(t_0)-u(t_0)$ at $t_0$. 
   From Propositions \ref{p:classic} and \ref{p:classic2} the term $\D_t^{\alpha} (v-u)(t_0)$ is well defined and 
    \[
     \int_{-\infty}^{t_0} [(v-u)(t_0) - (v-u)(s)]\K(t,s) \ ds \geq 0.
    \]  
   Since $(v-u)$ achieves a minimum over $(-\infty,t_0]$ at $t_0$ we also have 
    \[
     0 \geq \int_{-\infty}^{t_0} [(v-u)(t_0) - (v-u)(s)]\K(t,s) \ ds.    
    \] 
   Then if $(v-u)(t_0) \leq 0$ we have that $(v-u)(t)=0$ for all $t\leq t_0$. Then either $(v-u)(t_0)>0$ or 
   $(v-u)(t)=0$ for all $t\leq t_0$. 
  \end{proof}

  With a comparison principle available, we may use Perron's method to prove existence of solutions. 
  
  \begin{theorem}   \label{t:exist}
   Let $\phi(t)$ be continuous on $(-\infty,T_1]$ and $f$ continuous on $[T_1,T_2]$. There exists a unique viscosity solution $u$
   to 
    \[
     \begin{cases}
      \D_t^{\alpha} u(t) = f(t) &\text{ for } t\in [T_1,T_2] \\
      u(t)=\phi(t) &\text{ if } t\leq T_1
      \end{cases}
    \]
    on $(T_1,T_2]$.
  \end{theorem}
  
  The proof of Theorem \ref{t:exist} is standard. We forgo the proof since the ideas and methods are later used in the proof of Lemma \ref{l:zorn}. 
  One may also show existence for solutions of \eqref{e:main}; however, since the main focus of this paper is the regularity of solutions we do not include 
  the result here.

  Later in the paper we will need that the continuous divergence solutions constructed in \cite{acv16,acv17} are
  also viscosity solutions. 
  
  \begin{theorem}   \label{t:divnondiv}
   Let $f$ be a continuous function on $[-2,0]$ with $f(-2)=0$. 
   Assume the kernel $\K(t,s)$ of $\D_t^{\alpha}$ satisfies \eqref{e:timediv} and \eqref{e:alphabound}. 
   Then the weak divergence solution $v$ constructed in \cite{acv16} of  
    \begin{equation}   \label{e:divint1}
     \begin{cases}
     \D_t^{\alpha} v = f  &\text{ in } (-\infty,0].    \\
     v=0                        &\text{ for } t\leq -2,
     \end{cases}
    \end{equation}
    is also a viscosity solution and hence the unique viscosity solution.   
  \end{theorem}
   
  \begin{proof}
   In \cite{acv16,acv17} the solution $v$ is constructed with an $\e$ approximation.
   Via recursion we find the solution $v_{\e}$ to  
    \begin{equation}   \label{e:discrete}
     \e \sum_{i<j} [v_{\e}(\e j) - v_{\e}(\e i)]\K(\e j, \e i) = f(\e j),
    \end{equation}
   where $i,j \in \Z$. From \cite{acv16,acv17}, $v_{\e} \to v$ where $v$ solves the divergence form equation
    \begin{equation}  \label{e:divlimit}
     \begin{aligned}
      &\int_{-\infty}^T \int_{-\infty}^t [v(t)-v(s)][\phi(t)-\phi(s)]\K(t,s) \ ds \ dt \\
      & \ + \int_{-\infty}^T \int_{-\infty}^{2t-T}v(t)\phi(t)\K(t,s) \ ds \ dt 
       - \int_{-\infty}^T v(t)\D_t^{\alpha} \phi(t) \ dt \\
      &= \int_{-\infty}^T f(t)\phi(t) \ dt, \\
     \end{aligned}
    \end{equation}
    for all $T\leq 0$ and $\phi$ bounded and Lipschitz on $(-\infty,0]$. 
    We extend $v_{\e}$ to all of $(-\infty,0]$ by $v_{\e}(t)=v(\e i)$ where $\e (i-1) < t \leq \e j$.
    We now show that $v$ is a viscosity solution. The main idea in the following computations is that since \eqref{e:discrete}
    is a discrete equation, then $v_{\e}$ is also a viscosity-type discretized solution, and so the limit
    will also be a viscosity solution. 
    We now let $\psi$ be a Lipschitz function touching $v$
    strictly from above at $t_0 \leq 0$, that is $v(t)<\psi(t)$ for $t<t_0$ and $v(t_0)=\psi(t_0)$. 
    We let $\delta_{\e}$ be such that $\psi+\delta_{\e}$ touches $v_{\e}$ from above at $t_{\e} \leq t_0$. 
    Since $\psi$ touches $v$ strictly from above at $t_0$, then $t_{\e} \to t_0$ as $\e \to 0$. 
    Now If $\e (j_{\e}-1)< t_{\e} \leq \e j_{\e}$, then since $\psi + \delta_{\e}$ touches $v_{\e}$ from above and 
    \eqref{e:discrete} is a discrete equation, we have that 
     \[
      \e \sum_{i<j_{\e}} [\psi(t_{\e}) - \psi(\e i)]\K(\e j_{\e}, \e i) \geq f(\e j_{\e}).
     \]
     Now
     \[
      \begin{aligned}
      \e \sum_{i<j_{\e}} [\psi(t_{\e}) - \psi(\e i)]\K(\e j_{\e}, \e i) &=
       \e \sum_{i<j_{\e}} [\psi(t_{\e}) - \psi(\e j_{\e})]\K(\e j_{\e}, \e i)\\
       &\quad + \e \sum_{i<j_{\e}} [\psi(\e j_{\e}) - \psi(\e i)]\K(\e j_{\e}, \e i) \\
       &= (I) + (II). 
      \end{aligned}
     \]
     Since $\psi$ is Lipschitz we have
      \[
       \begin{aligned}
       |(I)| &= \left| \e \sum_{i<j_{\e}} [\psi(t_{\e}) - \psi(\e j_{\e})]\K(\e j_{\e}, \e i) \right| \\
             & \leq \e  \sum_{i<j_{\e}} [C \e ]\K(\e j_{\e}, \e i) \\
             & \leq \e^2 \sum_{i<j_{\e}}\Lambda \e^{-1-\alpha} (j_{\e}-i)^{-1-\alpha} \\
             & = \e^{1-\alpha} \Lambda \sum_{i=1}^{\infty} i^{-1-\alpha}.
       \end{aligned}
      \]
     Then $|(I)| \to 0$ as $\e \to 0$. Since $\psi$ is Lipschitz continuous it follow that $(II) \to \D_t^{\alpha}(t_0)$
     as $\e \to 0$. Since $f$ is continuous and $t_{\e} \to t_0$ as $\e \to 0$, we have that 
     $f(\e j_{\e}) \to f(t_0)$. Then $\D_t^{\alpha} \psi(t_0) \geq f(t_0)$. The proof if $\psi$ touches from below is similar. 
     Then $v$ is a viscosity solution. 
  \end{proof}
  
  This next Corollary will be useful when we want to show a limit is not equivalently zero. 
  \begin{corollary}   \label{c:divint}
   Let $v$ be the viscosity solution to \eqref{e:divint1} as in Theorem \ref{t:divnondiv}.  Assume also that $v \geq 0$. Then
   for $-2 <t\leq 0$ we have 
    \begin{equation}  \label{e:divint2}
    \frac{\lambda\alpha}{\Gamma(1-\alpha)} \int_{-2}^{t} \frac{v(s)}{(t-s)^{\alpha}} \ ds 
    \leq \int_{-2}^{t} f(s) \ ds \leq  \frac{\Lambda\alpha}{\Gamma(1-\alpha)} \int_{-2}^{t} \frac{v(s)}{(t-s)^{\alpha}} \ ds.   
    \end{equation}
  \end{corollary}
  
  \begin{proof}
   We take $\phi \equiv 1$ in \eqref{e:divlimit}. 
   Then 
    \[
     \begin{aligned}
      \int_{-2}^{T} f(t) \ dt &= \int_{-\infty}^{T}\int_{-\infty}^{2t-T} v(t) K(t,s) \ ds \ dt \\
                                      &\leq \frac{\alpha \Lambda}{\Gamma(1-\alpha)} \int_{-\infty}^{2t-T} \int_{-\infty}^t \frac{v(t)}{(t-s)^{1+\alpha}} \ ds \ dt \\
                                     &=  \frac{\alpha\Lambda}{\Gamma(1-\alpha)} \int_{-2}^{T} \frac{v(t)}{(T-t)^{\alpha}} \ dt.   
     \end{aligned}
    \]
    The bound from the other side is shown in the same manner. 
  \end{proof}

\section{H\"older continuity for the time derivative}  \label{s:holder}
 As explained in the introduction, since the integration for the generalized Marchaud derivative only occurs from the left, the results contained in 
 \cite{krs14} and \cite{ss16} do not cover Theorem \ref{t:main}. 
 We will follow the ideas in \cite{cs09} to prove the H\"older continuity. The outline and statements of the Lemmas are intentionally similar to those
 in Sections $8$-$12$ in \cite{cs09} so that the reader may compare and contrast properties of the operator $\D_t^{\alpha}$ with properties of $\mathcal{L}$. 
 Rather than work with the concave envelope we will work with a different envelope. 
  \[
   \M_{u}(t) := \sup_{s\leq t} u(s). 
  \]
 
 \begin{lemma}  \label{l:meas}
  Assume $M_{\alpha}^- u(t) \leq f(t)$ and $\M_u(t_0)=u(t_0)$. Let $r_k :=2^{-1/(1-\alpha)}2^{-k}$ and $R_k(t):= [t-r_k,t-r_{k+1}]$.
  There is a constant $C_0$ depending on $\lambda,\Lambda$ but not $\alpha$ such that for any $M>0$, there exists $k$ such that
   \begin{equation}  \label{e:upbound}
    |R_k(t_0) \cap \{u(s)<u(t_0)-Mr_k\}|\leq C_0 \frac{f(t_0)}{M} |R_k(t_0)|.
   \end{equation}
 \end{lemma}
 
 \begin{proof}
  We first note that $|R_k(t_0)|=r_{k+1}=r_k/2$. Since $\M_u(t_0)=u(t_0)$, then $u$ is touched from above by the constant $\M_u(t_0)$. Then from Proposition \ref{p:classic2} we have
   \[
    \int_{-\infty}^{t_0} [u(t_0)-u(s)]K(t_0,s) ds \leq f(t_0). 
   \]
  Suppose that the conclusion is not true. It then follows that
   \[
    \begin{aligned}
     f(t_0) \geq M_{\alpha}^- u(t_0) &\geq \frac{\alpha \lambda}{\Gamma(1-\alpha)} \int_{-\infty}^{t_0} \frac{u(t_0)-u(s)}{(t_0 - s)^{1+\alpha}} ds \\
                              &\geq \frac{\alpha \lambda}{\Gamma(1-\alpha)} \int_{t_0-2^{-1/(1-\alpha)}}^{t_0} \frac{u(t_0)-u(s)}{(t_0 - s)^{1+\alpha}} ds \\                                     
                              &\geq \frac{\alpha \lambda}{\Gamma(1-\alpha)} \sum_{k=0}^{\infty} \int_{R_k(t_0)} \frac{u(t_0)-u(s)}{r_{k+1}^{1+\alpha}} ds \\
                                     &\geq \frac{\alpha \lambda}{\Gamma(1-\alpha)} M r_k \frac{C_0f(t_0)}{M} \frac{|R_k(t_0)|}{r_{k+1}^{1+\alpha}} \\
                                     &= 2^\alpha \frac{\alpha \lambda}{\Gamma(1-\alpha)} C_0 f(t_0) \sum_{k=0}^{\infty} r_k^{1-\alpha}.
    \end{aligned}
   \]
  Now
   \[
    \sum_{k=0}^{\infty} r_k^{1-\alpha} = \frac{1}{2}\sum_{k=0}^\infty\left(\frac{1}{2^{1-\alpha}} \right)^k = \frac{1}{2}\left[\frac{1}{1-2^{-(1-\alpha)}} \right].
   \]
  Then 
   \[
    f(t_0) \geq  \frac{\alpha \lambda}{\Gamma(1-\alpha)} \frac{C_0 f(t_0)}{2^{1-\alpha}-1}. 
   \]
  From the following relation on the Gamma function
   $ \Gamma(1-z)\Gamma(z)=\frac{\pi}{\sin(\pi z)} $ 
  it follows that for $\alpha_0<\alpha\leq 1$, the right hand side of the above equation is bounded below. Then for $C_0$ large enough we obtain a contradiction. 
 \end{proof}

 \begin{remark}  \label{r:alpha}
  In the above proof it is clear that by choosing $C_0$ larger (say for instance $2C_0$) that for fixed $\alpha$ and $M$ bounded from above 
  we may choose $k \leq N$ for some large $N$. This $N$ will
  depend on $\alpha$, but this gives a bound from below on $r_k$ such that \eqref{e:upbound} holds.  
 \end{remark}
 
 
 
  
 \begin{corollary}  \label{c:estimate}
  For any $\e_0>0$, there exists $C$ such that for $u$ as defined in Lemma \ref{l:meas}, there exists $r \in (0,2^{-1/(1-\alpha)})$ such that if 
  $\M_{u}(t)=u(t)$ and $M_{\alpha}^- u(t)\leq f(t)$, then there exists a constant $c_\alpha$ (depending on $\alpha$) and $r \in (c_{\alpha},2^{-1/(1-\alpha)})$
  such that 
   \begin{equation}  \label{e:pull}
    \frac{|\{s \in [t-r,t-r/2]\}: u(s)<u(t)-Cf(t)r|}{|[t-r,t-r/2]|} \leq \e_0
   \end{equation}
 \end{corollary}

 \begin{proof}
  By choosing $M=Cf(t)\e_0^{-1}$ we obtain \eqref{e:pull} from Lemma \ref{l:meas} and Remark \ref{r:alpha}. 
 \end{proof}
 
 \begin{lemma}  \label{l:abp}
  Let $u(t) \leq 0$ for $t<-2$ and $M_{\alpha}^- u \leq f$ on $[-2,T]$. There exists a constant $C$ depending on $\lambda, \alpha_0$ and 
  there exists finitely many intervals $I_j$ with disjoint interiors and with length between $c_\alpha$ and 
  $2^{-(1-\alpha)}$ such that 
   \begin{equation}  \label{e:abp}
    \sup u^+ \leq C \sum_{j} \max_{I_j}f^+(t) |I_j|.
   \end{equation}
  Furthermore, if $I_j=[t_j-r_j,t_j]$ is such an interval, and $[t_j-2r_j,t_j]\subset [-2,T]$, then 
   \begin{equation}  \label{e:abp2}
    |\{s \in [t_j-2r_j,t_j-r_j]:u(s)\geq u(t_j)-Cr_j\}|\geq \mu r_j. 
   \end{equation} 
 \end{lemma}
 
 \begin{remark}
  Notice that as $\alpha \to 1$ the inequality \eqref{e:abp} becomes
   \[
    \sup u \leq C \int_{\{u=\M_u\}} f^+(s) ds.  
   \] 
 \end{remark}
 
 \begin{proof}
  We will choose the finite sequence of intervals inductively. 
  Let $u^+$ achieve its maximum on $[-2,T]$ at the point $t_1$. From Corollary \ref{c:estimate} there exists $r_1$ with 
  $c_{\alpha}\leq r_1\leq 2^{-{1-\alpha}}$ such that \eqref{e:abp2} holds for $I_1 = [t_1 - r_1,t_1]$.
   \[
     \M_{u^+}(t_1)-\M_{u^+}(t_1-r_1) \leq u(t_1) - u(s)  \text{  for any  } s \in [t_1-2r_1,t_1-r_1].
   \]
   Since $[t_1-2r_1,t_1-r_1]$ satisfies \eqref{e:abp2}, we then have that 
    \[
     \M_{u^+}(t_1) - \M_{u^+}(t_1-r_1) \leq Cf(t_1)r_1 = Cf(t_1)|I_1|.
    \]
    Proceeding inductively, supposing $I_{j-1}$ has been chosen, we choose 
     \[
      t_{j} := \max \{-2\leq s \leq t_{j-1}-r_{j-1}: u(s)=\M_{u}(s)\}.
     \]
    If no such $t_j$ exists then the process terminates. Otherwise, by Corollary \ref{c:estimate} we 
    choose $r_j \in (c_{\alpha},2^{-1/(1-\alpha)})$ such that $[t_j-2r_j,t_j-r_j]$ satisfies \eqref{e:abp2}. 
    We label $I_j = [t_j-r_j,t_j]$. As in the case of $I_1$ we obtain 
     \[
       \M_{u^+}(t_j)-\M_{u^+}(t_j-r_j) \leq Cf(t_j)|I_j|. 
     \]
    Since each $r>c_{\alpha}$ the process will eventually terminate. 
    Now
     \[
      \{u = \M_{u^+}\} \subseteq \cup_{j} I_j,
     \]
    and so
     \[
      \sup u^+  \leq  \sum_{j} |\M_{u^+}(t_j) - \M_{u^+}(t_j - r_j) \leq C\sum_{j} \max_{I_j} f^+(t) |I_j|.
     \]
    Finally, from how each $I_j$ is chosen we have \eqref{e:abp2} holds for each $[t_j-2r_j,t_j-r_j]$ and also $\mathring{I}_j \cap \mathring{I}_k = \emptyset$ 
    for $j\neq k$.    
 \end{proof}
 
 For this next Lemma for an interval $I_j = [t_j - r_j, t_j]$ we define
  \[
   2I_j := [t_j - 2r_j, t_j]. 
  \]
 \begin{lemma}   \label{l:cover}
  Let $A \subset [0,1]$ and let $c_1,c_2$ be two positive constants. Let 
  $\{I_j: j \in \mathcal{I}\}$ be a collection of closed intervals with nonempty interior satisfying the following for every $j \in \mathcal{I}$:
   \[
    \begin{aligned}
     &(i) \quad 2I_j \subset [0,1]  \\
     &(ii) \quad \dot{I}_i \cap \dot{I}_j = \emptyset \text{ for every } i \neq j \\
     &(iii) \quad \sum_{j \in \mathcal{I}} |I_j| > c_1 \\
     &(iv) \quad |A \cap 2I_j| \geq c_2 |2I_j|,
    \end{aligned}
   \]
  then  $|A| \geq (c_1 \cdot c_2)/3$. 
 \end{lemma}
 
 \begin{proof}
  We first assume the collection $\{I_j\}$ is finite.
  We will choose a subset $\mathcal{J} \subseteq \mathcal{I}$.  
   Recall that each interval has nonempty interior and the 
  intersection of the interiors is empty. We will pick a subcollection of intervals as follows: We choose and label
  $I_1$ as the interval with the farthest right end point. If there is another interval $I_k$ for some 
  $k \in \mathcal{I}$ such that 
   \[
    \begin{aligned}
     &(1) \quad  2I_k \cap 2I_1 \neq \emptyset \text{  and  } \\
     &(2) \quad  2I_k \setminus (2I_k \cap 2I_1) \neq \emptyset,
    \end{aligned}
   \]
  then we choose $I_2$ to 
  be the interval satisfying $(1)$ and $(2)$ such that $2I_k$ has the farthest left endpoint. If no $2I_k$
  satisfies $(1)$ and $(2)$ then we choose $I_2$ to be the interval whose right end point is closest to the left 
  end point of $2I_1$. We then choose all remaining intervals in the same manner. It is clear from the construction
  that each interval from the collection $\{2I_j\}$ for $j \in \mathcal{J}$ can intersect at most two other intervals
  from the same collection. It is also clear that 
   \begin{equation}  \label{e:covered}
    \bigcup_{i \in \mathcal{I}} I_i \subset \bigcup_{j \in \mathcal{J}} 2I_j,
   \end{equation}  
  Then
   \[
    \begin{aligned}
     |A| &\geq \frac{1}{2} \sum_{j \in \mathcal{J}} |A \cap 2I_j| \\
         &\geq \frac{c_2}{3} \sum_{j \in \mathcal{J}} |2I_j| \\
         &\geq \frac{c_2}{3} \sum_{i \in \mathcal{I}} |I_i| \\
         &\geq \frac{c_1 c_2}{3}. 
    \end{aligned}
   \]
  If the collection $\{I_i\}$ is infinite, then we simply choose a finite subcollection of intervals 
  whose measure is within $\e$, and then let $\e \to 0$. 
 \end{proof}

 We now prove the most important Lemma. 
 
 \begin{lemma}  \label{l:krylov}
  Let $0<\alpha_0\leq \alpha \leq 1$. There exists $\e_0 >0$, $0<\mu<1$, and $M>1$ all depending on $\alpha_0, \lambda, \Lambda$ such that if 
   \[
    \begin{aligned}
     (i) \quad &u \geq 0 \text{ in } (-\infty,1] \\
     (ii) \quad &u(1) \leq 1/2 \\
     (iii) \quad &M_{\alpha}^+ u \geq -\e_0 \text{ in } [-1,1],
    \end{aligned}
   \]
  then 
   \[
    |\{u \leq M\} \cap [0,1]|>\mu. 
   \]
 \end{lemma}

 \begin{proof}
  We utilize the function 
   \[
    \Phi(t):=
    \begin{cases}
     0 &\text{ for } t\leq 3/4 \\
     4(t-3/4) &\text{ for } t \geq 3/4.
    \end{cases}
   \]
  Notice that $M_{\alpha}^- u \leq  4\lambda$ for all $0<\alpha\leq 1$. We now apply Lemma \ref{l:abp} to $v:= \Phi - u$. Since $u(1)\leq 1/2$ we have
  that $v(1)\geq 1/2$. We have From estimate \eqref{e:abp} 
  we have
   \[
    \begin{aligned}
    1/2 \leq \sup v &\leq C \sum_{j}  \max_{I_j} M_{\alpha}^- (\Phi - u) |I_j|  &\text{ from } \eqref{e:abp}\\
                    &\leq C(\e_0 + 4\lambda C_1)\sum_{j} |I_j|      &\text{ from Proposition } \ref{p:basics}   \\
                    &\leq C\e_0 + 4\lambda C_1\sum_{j} |I_j|.     
    \end{aligned}
   \]
  Then for $\e_0$ small enough  we obtain
   \[
    \frac{1}{4}\leq C \sum_{j} |I_j|. 
   \]
  Now $\Phi$ is supported in $[3/4,1]$. Then if $I_j=[t_j-r_j/4, t_j]$, then $t_j \in [3/4,1]$ and so $2I_j :=[t_j - r_j,t_j]\subset [0,1]$. 
  From \eqref{e:abp2}
   \[ 
    |\{s\in 2I_j : v(s)\geq v(t_j) - Cr_j \}| \geq \mu |I_j|. 
   \]
  We recall that $r_j \leq 2^{-(1-\alpha)}$ and  $v(t_j) \geq 0$, so that
   \[
    |\{s\in 2I_j : \Phi(s)- u(s) \geq  - C \}| \geq \mu |I_j|.
   \]
  Furthermore, $\Phi(s)\leq 1$, and so 
   \[
    |\{s\in 2I_j : 1+ C  \geq u(s) \}| \geq \mu |I_j|.
   \]
  We now use Lemma \ref{l:cover} to conclude that 
   \[
    |\{s \in [0,1] : 1+C \geq u(s) \}| \geq \mu_1. 
   \]
 \end{proof}

 Using Lemma \ref{l:krylov} with Lemma 4.2 from \cite{cc95} one proves the following two Lemmas just as 
 in \cite{cc95}. 
 
 \begin{lemma}  \label{l:krylov2}
  Let $u$ be as in Lemma \ref{l:krylov}, then 
   \[
    |\{u > M^k\} \cap [0,1]|\leq (1-\mu)^k. 
   \]
 \end{lemma} 
 
 \begin{lemma}  \label{l:krylov3}
  Let $u$ be as in Lemma \ref{l:krylov}, then 
   \[
    |\{u > t\} \cap [0,1]|\leq dt^{-\e} 
   \]
  where $d,\e>0$ depend only on $\lambda, \Lambda, \alpha_0$. 
 \end{lemma}
 
 We now have
 
 \begin{theorem}  \label{t:Lepsilon}
  Let $u \geq 0$ in $(-\infty,0]$ and $M_{\alpha}^+ u \geq -C_0$ in $[-2r,0]$. Assume $\alpha \geq \alpha_0>0$. Then 
   \[
    |\{u >t\}\cap [-r,0]|\leq dr(2u(0)+C_0r^{\alpha}\e_0^{-1})^{\e} t^{-\e} \text{  for every  } t
   \]  
  where the constants $d,\e$ depend only on $\lambda, \Lambda,$ and $\alpha_0$ as in Lemma \ref{l:krylov3}. 
 \end{theorem}
 
 \begin{proof}
  We rescale and let 
   \[
    v(t):= \frac{u(rt)}{2u(0)+C_0r^{\alpha} \e_0^{-1}}. 
   \]
  Then $M_{\alpha}^+ v \geq -\e_0$ on $[-2,0]$ and $v(0)\leq 1/2$, so we may apply Lemma \ref{l:krylov3} to $v$ and conclude 
   \[
    |\{v > t\} \cap [-1,0]|\leq dt^{-\e}, 
   \]
  which written in terms of $u$ becomes 
   \[
    |\{u>t\} \cap [-r,0]| \leq dr[2u(0)+C_0r^{\alpha}\e_0^{-1}]^{\e}t^{-\e}. 
   \]
 \end{proof}
 
 Theorem \ref{t:main} will follow from the following
 \begin{theorem}   \label{t:regular}
  Let $\alpha_0 < \alpha <1$ for some $\alpha_0>0$. Let $u$ be a bounded in $(-\infty,0]$ and continuous on $[-1,0]$. If 
   \[
    \begin{aligned}
     M_{\alpha}^+ u &\geq - C_0 \text{ in } (-1,0], \\
     M_{\alpha}^- u &\leq C_0   \text{ in } (-1,0],
    \end{aligned}
   \]
   then there is a $\beta>0$ and $C>0$ both depending only on $\lambda, \Lambda, \alpha_0$ such that $u \in C^{0,\beta}([-1/2,0])$ and 
     \[
      \| u \|_{C^{0,\beta}([-1/2,0])} \leq C(\| u \|_{L^{\infty}} + C_0).
     \]
 \end{theorem}
 
 Theorem \ref{t:regular} follows from Theorem \ref{t:Lepsilon} and the following Lemma
 
 \begin{lemma}  \label{l:main}
  Let $-1/2 \leq u \leq 1/2$ in $(-\infty,0]$. There exists $\delta_0$ depending only on $\lambda, \Lambda$ and $0<\alpha_0<1$ such that if 
   \[
    \begin{aligned}
     &M_{\alpha}^+ u \geq -\delta_0 &\text{ in } [-1,0]\\
     &M_{\alpha}^- u \leq  \delta_0 &\text{ in } [-1,0],
    \end{aligned}
   \]
  then there is a $\beta$ depending only on $\lambda, \Lambda,$ and $\alpha_0$ such that if $1>\alpha\geq \alpha_0$
   \begin{equation}  \label{e:beta1}
    |u(t)-u(0)| \leq C|t|^\beta
   \end{equation}
  For some constant $C$. 
 \end{lemma}
 
 \begin{proof}
  The proof is nearly the same as in \cite{cs09}. We give the details. We will construct a sequence $m_k, M_k$ with $m_k \leq u \leq M_k$ and 
   \begin{align}  \label{e:beta}
    m_k \leq u \leq M_k \text{  in  } [-4^{-k},0]  \\
    m_k \leq m_{k+1} \text{  and  } M_k \geq M_{k+1}  \\
    M_k - m_k = 4^{-\beta k}.
   \end{align}
  This will show \eqref{e:beta1} with constant $C=2^{\beta}$. For $k=0$ we choose $m_0=-1/2$ and $M_0=1/2$ and by assumption we have \eqref{e:beta}. 
  We proceed by induction. Assume \eqref{e:beta} holds up to $k$. 
  We have either 
    \begin{gather*}
    \left |\{u \geq (M_k + m_k)/2\} \cap [-3\cdot 4^{-(k+3)},-4^{-(k+2)}] \right| \geq 4^{-(k+4)} \\
    \text{  or  }  \\
    \left |\{u \leq (M_k + m_k)/2\} \cap [-3\cdot4^{-(k+3)},-4^{-(k+2)}] \right| \geq 4^{-(k+4)}.
   \end{gather*}
  If we assume the first then we define
   \[
    v(t):= \frac{u(4^{-k}t)-m_k}{(M_k - m_k)/2}.
   \]
  We have that $v \geq 0$ in $[-1,0]$ and $|\{v\geq 1\} \cap [-1/2,-1/4]| \geq 1/16$. We also have that
   \[
    M_{\alpha}^+ v \geq \frac{-4^{-k\alpha}\delta_0}{(M_k - m_k)/2} = -2\delta_0 4^{k(\beta-\alpha)} \geq -2\delta_0
   \]
  as long as $\beta <\alpha$. From the inductive hypothesis, for any $0 \leq j < k$ we have 
   \[
    v \geq \frac{m_{k-j} - m_k}{2(M_k-m_k)} \geq \frac{m_{k-j}-M_{k-j}  + M_{k}- m_k}{2(M_k-m_k)} \geq 2(1-4^{\beta j}),
   \]
  and so $v \geq -2(|4t|^{\beta}-1)$ outside $[-1,0]$. We define $w:= \max\{v,0\}$. For any $t \in [-3/4,0]$ we have
    \begin{align} \label{e:remark}
     M_{\alpha}^+ w(t) - M_{\alpha}^+v(t) &= \frac{\alpha}{\Gamma(1-\alpha)}\int_{\{v<0\}\cap \{t<-1\}} \frac{\Lambda v(s)}{(t-s)^{1+\alpha}}  \\
                        &\geq \frac{\alpha \Lambda}{\Gamma(1-\alpha)}\int_{-\infty}^{-1} \frac{-2(|4s|^\beta-1)}{(t-s)^{1+\alpha}}. 
    \end{align}         
  Thus $M_{\alpha}^+ w - M_{\alpha}^+ v \geq -2 \delta_0$ for $\beta$ small enough. 
  We then have $M_{\alpha}^+ w \geq -4\delta_0$, and so for any $t_0 \in [-1/4,0]$ we may apply Theorem \ref{t:Lepsilon} to $w$ 
  to obtain
   \[
    \begin{aligned}
    \frac{1}{16} \leq |\{w>1\}\cap[-3/8,-1/4]| &\leq |\{w>1\}\cap [-3/8,t_0]| \\
                                                                 &\leq d(1/8)(2u(t_0) +  4(1/8)^{\alpha} \delta_0 e_0^{-1})^{\e}.
    \end{aligned}
   \]
  Thus, for $\delta_0$ small enough we have $u(t_0)\geq \theta>0$ for any $t_0 \in [-1/4,0]$. We let $M_{k+1}=M_k$ and $m_{k+1}=m_k + \theta(M_k-m_k)/2$.
  Then, $M_{k+1}-m_{k+1}=(1-\theta/2)4^{-\beta k}$. Choosing $\beta$ and $\theta$ small with $(1-\theta/2)=4^{-\beta}$ we have 
  $M_{k+1}-m_{k+1}= 4^{-\beta (k+1)}$. 
  
  If on the other hand $|\{u\leq (M_k +m_k)/2 \cap [-3\cdot 4^{-(k+3)},-4^{-(k+2)}]\}| \geq 4^{-(k+4)}$, then we define
   \[
    v:= \frac{M_k - u(4^{-k}t)}{(M_k - m_k)/2}
   \]
  and use that $M_{\alpha}^- u \leq \delta_0$. 
 \end{proof}

 \section{A Barrier}    \label{s:barrier}
  In this section we construct a subsolution to an ordinary differential equation that will allow us to prove the H\"older continuity. 
  We begin with the following 
   \begin{lemma}  \label{l:cont}
    Let $u \in C^{0,\beta}((-\infty,0])$ with $ \alpha< \beta <1$. 
    Then $M_{\alpha}^{\pm} u$ is continuous on $(-\infty,0]$.  
   \end{lemma}
 
   \begin{proof}
    We will show that for fixed $t$  and any $\epsilon >0$, there exists $h_0$  such that if $0\leq h <h_0$, then 
     \begin{equation} \label{e:lipcont}
      \left|M_{\alpha}^{\pm} u(t+h) - M_{\alpha}^{\pm}u(t) \right| < \epsilon. 
     \end{equation}
     For fixed $\xi<t$ and for $0\leq h$, we have 
    \[
     \begin{aligned}
     &\left| \frac{\alpha}{\Gamma(1-\alpha)} \int_{\xi}^{t+h} \frac{\Lambda [u(t+h)-u(s)]_+ -\lambda[u(t+h)-u(s)]_-}{(t+h-s)^{1+\alpha}} \ ds \right|  \\
      &\quad \leq \frac{\alpha}{\Gamma(1-\alpha)} \Lambda \| u\|_{C^{0,\beta}} \frac{(t+h-\xi)^{\beta-\alpha}}{\beta-\alpha}.
     \end{aligned}
    \]
     We may choose $\xi$ close enough to $t$ and choose $h$ small enough, so that the above inequality is less than $\epsilon/2$. 
     Now since $u$ is continuous, if $t>\xi$, then both
     \[
          \left|\int_{-\infty}^{\xi}  \frac{ [u(t+h)-u(s)]_-}{(t+h -s)^{1+\alpha}}      
      -\frac{[u(t)-u(s)]_-}{(t-s)^{1+\alpha}} \ ds \right| \to 0
      \]
      and
      \[
     \left| \int_{-\infty}^{\xi}  \frac{ [u(t+h)-u(s)]_+}{(t+h -s)^{1+\alpha}}      
      -\frac{[u(t)-u(s)]_+}{(t-s)^{1+\alpha}} \ ds \right| \to 0
    \]     
    as $h \to 0$.  Then we may choose $h$ small enough, so that  \eqref{e:lipcont} holds. 
     \end{proof}

   For the next Lemma we will require a solution to an ordinary differential equation. 
   We recall that from Theorem \ref{t:exist} if  $f$ is continuous on $[-2,0]$ with $f(-2)=0$. 
   Then exists a viscosity solution $u$ to the differential equation
     \[
      \begin{cases}
       u(t)=0 &\text{ for } t\leq -2\\
       \D_t^{\alpha} u = f(t) &\text{ on } (-\infty,0]. 
      \end{cases}     
     \]
  That $u$ is a viscosity solution on $(-2,0]$ is a direct result from Theorem \ref{t:exist}. Since $f(t)=u(t)=0$ for $t\leq -2$
  it is immediate that $u$ is also a solution on $(-\infty,-2]$ as well. From Theorem \ref{t:main} the solution $u$ is H\"older continuous. 
 
  In the next Lemma we will also utilize the following bump function. 
  Let $\eta \geq 0$ with support in $[0,1]$. Let $\eta' \geq 0$ for $t \leq 1/2$ and $\eta' \leq 0$ for $t\geq 1/2$. 
  We note that $\D_t^{\alpha} \eta \leq C$ for some $C$ independent of $\alpha$ and $\D_t^{\alpha} \eta(t) <0 $ for $t>1$. We will use 
   \[
    \eta_{\epsilon} := \epsilon \eta(t/\epsilon).  
   \] 
 
  \begin{lemma}  \label{l:zorn}
  Let $f$ be smooth on $[-2,0]$ with $f(-2)=0$. Let $u$ be the viscosity solution to 
    \begin{equation}  \label{e:subbelow}
       \begin{cases}
       \D_t^{\alpha}  u(t) = f  &\text{ on } (-\infty,0] \\
       u(t)=0    &\text{ for } t\leq -2.
       \end{cases}  
     \end{equation}
 Then there exists a sequence of Lipschitz subsolutions $\{u_{k}\}$ to the above equation with $u_{k}\leq u$ and $u_{k}\to u$ uniformly on $[-2,0]$.  
 \end{lemma}
 
 \begin{proof}
  We consider for fixed $M> \sup f$  the set 
   \[
    \mathfrak{K}_1 := \{w: -M \leq D_t^{\alpha} w \leq f \text{ on } [0,2], \text{ and } w \text{ is Lipschitz}\}.
   \]
  We define 
 \[
    \mathfrak{K}_2 := \{v : -M \leq \D_t^{\alpha} v \leq f \text{ on } [0,2] \text{ and } u_k \Rightarrow v \text{ with } u_k \leq v \text{ and } u_k \in \mathfrak{K}_2 \}.   
 \]
  We note that any $v \in \mathfrak{K}_2$ is continuous since it is the uniform limit on a compact set of Lipschitz continuous functions. 
   
  We now show that $\mathfrak{K}_1$ is nonempty, and hence $\mathfrak{K}_2$ is also nonempty. From the theory of ordinary differential equations
  \cite{d04}, we solve 
   \[
    \partial_t^{\alpha} g(t) = h(t)/\lambda \leq f(t) /\lambda, 
   \]
  with $h$ smooth and strictly decreasing.  Then $g$ is strictly decreasing, and so 
   \[
    \D_t^{\alpha} g \leq M_{\alpha}^+ g(t) = \lambda \partial_t^{\alpha} g(t) = h(t) \leq f(t).
   \]  
  Thus $g$ is a smooth subsolution to \eqref{e:subbelow} and $g$ can be approximated from below by itself. We choose $-M< \inf h/\lambda$. 
  Thus $\mathfrak{K}_2$ is nonempty.
  
  We now assign a partial ordering to $\mathfrak{K}_2$ with the natural assignment that $v_1 \leq v_2$ if $v_1\leq v_2$ everywhere on 
  $(-\infty,0]$. From the comparison principle we have that $u$ as in \eqref{e:subbelow}
  is an upper bound for $\mathfrak{K_2}$.    
  
  We will now show that $u \in \mathfrak{K_2}$. 
  By Zorn's Lemma there exists a maximal element $w$. We will show that $w \equiv u$. If $\D_t^{\alpha} w \equiv \D_t^{\alpha} u$, 
  then by comparison and uniqueness, $w \equiv u$. Suppose by way of contradiction that $w$ is not identically $u$. 
  Then $w$ is not a supersolution of \eqref{e:subbelow}, and so there exists a $t_0 \in [0,2]$ and a Lipschitz function $\psi$ with $\psi(t_0)=w(t_0)$ and 
  $\psi \leq w$ on $[t_0 - \delta,t]$ for some $\delta >0$ such that $\D_t^{\alpha} \phi(t_0) < f(t_0)$ where 
   \[
    \phi(t) :=
     \begin{cases}
       \psi(t) \text { if } t \in (t_0 - \delta],\\ 
        w(t)  \text{ if } t < t_0 - \delta.
     \end{cases}    
   \]
   Then from Proposition \ref{p:classic2} we can evaluate $D_t^{\alpha} w(t_0)$ classically and 
  $\D_t^{\alpha} w(t_0) < f(t_0)$. 
   Because $w \in \mathfrak{K}_2$ there exists Lipschitz subsolutions $w_{k} \to w$ uniformly from below. Then there exists $k$ large enough so that 
    \[
    \phi_k(t) :=
     \begin{cases}
       \psi(t) \text { if } t \in (t_0 - \delta],\\ 
        w_k(t)  \text{ if } t < t_0 - \delta.
     \end{cases}    
   \]   
   satisfies $M_{\alpha}^+ \phi_k < f$ on $[t_0 - \delta_1,t_0]$ for some $0<\delta_1 < \delta$. 
   We now consider two different situations. If $w_k(t_0)=w(t_0)$, then for $k$ large enough, 
   $\D_t^{\alpha}w_k(t_0) < f(t_0)$ because $\D_t^{\alpha} w(t_0)< f(t_0)$ classically. 
   Since $w_k$ is Lipschitz, it follows from Lemma \ref{l:cont} that $\D_t^{\alpha} w_k < f$ in 
   $[t_0 - \delta_2, t_0 + \delta_2]$. Then since $\D_t^{\alpha} (w_k + \eta_{\e}) = \D_t^{\alpha} w_k + \D_t^{\alpha} \eta_{\e}$, 
   then $w_k + \eta_{\e}$ is a Lipschitz subsolution for $\epsilon$ small enough, 
   and $w_k(t_0) + \eta_{\e}(t_0) > w(t_0)$. 
   Now the $\max\{w_k +\eta_{\e},w\} \in \mathfrak{K}_2$, and this contradicts the maximality of $w$.

   If in the second situation, $w_k(t_0) < w(t_0)$, then we extend $\phi_k$ to the right of $t_0$ by $\phi_k(t_0) = -M_1 (t_0 -t) + w(t_0)$. 
   From Lemma \ref{l:cont}, for $M_1$ large enough,  $\D_t^{\alpha} \phi_k \leq  f$ for all $t \in [-2,0]$, and $\D_t^{\alpha} \phi_k < f$ in $[t_0 - \delta_3, t_0 + \delta_3]$. 
   We let  $\tilde{\phi}_k := \max \{\phi_k, w_k\}$ and note that
   $\tilde{\phi}_k \in \mathfrak{K}_1$ and Lipschitz, and $\D_t^{\alpha} \phi_k < f$ in $[t_0 - \delta_4, t_0 + \delta_4]$ for some $\delta_4 >0$ with $\tilde{\phi}_k(t_0)=w(t_0)$. 
   Then as before we may take $\tilde{\phi}_k + \eta_{\epsilon}$  
   with $\epsilon$ small enough and obtain a contradiction.       
 \end{proof}

 In order to prove H\"older continuity of solutions to \eqref{e:nontime} we will follow the method presented in \cite{s11}. One of the main ingredients is to
 solve an ordinary differential equation in time. We begin with the following. 
 
 \begin{lemma} \label{l:timeexist}
  Let $C_1$ be a fixed constant. Let $g(t)$ be a continuous function on $[-2,0]$. There exists a continuous viscosity solution $m(t)$ in $[-2,0]$ to 
   \[
    \D_t^{\alpha} m(t)= C_1m(t) + g(t), 
   \]
  with $m(t)=0$ for $t\leq -2$. 
 \end{lemma}
 
 \begin{proof}
  Let $\beta_1 < \beta$ for $\beta$ as in Theorem \ref{t:main}. From Theorem \ref{t:exist}, for $v \in C^{\beta_1}([0,2])$ there exists a solution $h(t)$ to 
   \[
    \D_t^{\alpha} h(t)= C_1v(t) + g(t),
   \]
  with $h(t)=0$ for $t\leq -2$. From Theorem \ref{t:main} we have that
   \[
    \| h \|_{C^{0,\beta}} \leq C \| C_1 v + g \|_{L^{\infty}}. 
   \]
  Since $C^{0,\beta_1}([-2,0])$ is compactly contained in $C^{0,\beta}([-2,0])$ we have a compact mapping $M: v \to h$ from $C^{0,\beta_1}([-2,0])$ into itself. From 
  Corollary 11.2 in \cite{gt01} it follows that there is a fixed point $m(t)$ which is a viscosity solution from Lemma \ref{l:limit}.
 \end{proof}
 
 As in \cite{s11} we will utilize an ordinary differential equation to capture information backwards in time. 
 We consider the fractional ordinary differential equation 
    \begin{equation}  \label{e:fode}
     \begin{cases}
      m(-2)=0 & \text{ for } t\leq -2 \\
      \D_t^{\alpha}  m(t) = c_0 |\{x \in B_1: u(x,t)< 0\}| - C_1m(t) &\text{ for } -2<t\leq 0.
     \end{cases}
    \end{equation}
   We would like to use $m$ as a test function for a viscosity solution. However, since the right hand side is not continuous, 
   we cannot apply Lemma \ref{l:timeexist} to obtain the existence of $m$. Furthermore, the solution $m$ may not be Lipschitz and therefore 
   not a valid test function. To overcome these two issues we obtain a Lipschitz subsolution to \eqref{e:fode}. 
  We consider $|\{x \in B_1: u(x,t)< 0\}|$ rather than $|\{x \in B_1: u(x,t)\leq 0\}|$ because 
  we may easily approximate the former from below by smooth functions. We accomplish this by 
   considering $g(x,t,\e):=\min\{\epsilon^{-1}\max\{0,-u\} , 1\}$. We then let  
    \[
      G(t,\e) = \int_{B_1}g(x,t,\e) \ dx.
    \]
    Now $G(t,\e)$ is continuous in $t$, and $0\leq G(t,\e)\leq |\{x \in B_1: u(x,t)< 0\}|$, and $G(t,\e) \to |\{x \in B_1: u(x,t)< 0\}|$ as $\e \to 0$. 
    Since $G(t,\e)$ is continuous, from Lemma \ref{l:timeexist}, we may solve 
    \[
     \begin{cases}
      m(-2)=0 & \text{ for } t\leq -2 \\
      \D_t^{\alpha}  m(t) = c_0 G(t,\e) - C_1m(t) &\text{ for } -2<t\leq 0.
     \end{cases}
    \]


 \begin{lemma}  \label{l:mubound}
  Let $\alpha_0 \leq \alpha <1$. Assume that 
  the kernel for $\D_t^{\alpha}$ satisfies \eqref{e:timediv}. 
  Let $m$ be a solution to $\D_t^{\alpha} m = c_0 f(t)-C_1 m$ with $m(t)=0$ for $t\leq -2$, $0\leq f\leq |B_1|$,  and 
    \begin{equation}   \label{e:muf}
     \int_{-2}^{-1} f(t) \geq \mu. 
   \end{equation}
   Then there exists two moduli of continuity $\omega_1(c_0 \mu), \omega_2(C_1^{-1})$ with $\omega_i$ increasing and $\omega_i(s)>0$ for $s>0$ and depending only 
   on $\lambda, \Lambda, |B_1|, \alpha_0$ so that 
    \begin{equation}  \label{e:expbound}
     m(t) \geq \omega_1 (c_0 \mu)  \omega_2 (C_1^{-1}) \quad \text{for } -1\leq t \leq 0.
    \end{equation}
  \end{lemma}
 
  \begin{proof}
   From Lemma \ref{l:main} any such solution $m$ is H\"older continuous. We claim that
   $m>0$ for 
   $t \in [-1,0]$. Suppose by way of contradiction that there exists $t_0 \in [-1,0]$
   such that $m(t_0)\leq 0$. Let $m$ achieve its minimum at $t_1$. Then $m$ is touched 
   from below at $t_1$ by the constant function $m(t_1)$, and so by Propositions \ref{p:classic}
    and \ref{p:classic2} we may evaluate $\D_t^{\alpha} m$ at $t_1$ and  
     \[
      \int_{-\infty}^{t_1} [m(t_1) - m(s)] \K(t,s) \ ds \geq f(t_1) \geq 0.
     \]
     But since $m(t_1)\leq 0$ and $m$ achieves a minimum at $t_1$ we have
      \[
      0 \geq      \int_{-\infty}^{t_1} [m(t_1) - m(s)] \K(t,s) \ ds.
      \]
     Then $m(t)=0$ for $t\leq t_1$. But then since $f(t)=0$ for $t\leq t_1$. But this contradicts
     the assumption \eqref{e:muf}. Then $m(t)>0$ for $-1\leq t \leq 0$. 
   
   Let us fix $c_0,C_1, \mu >0$. 
   Suppose by way of contradiction that there exists $m_k, \alpha_k, f_k$ all satisfying the assumptions of the Lemma, but 
    \[
     \inf_{[-1,0]} m_k \to 0 
    \]
    as $k \to \infty$. From Lemma \ref{l:main}, we have $m_k \to m_0$ in $C^{0,\beta_1}([-2,0])$ for any $\beta_1 < \beta$. 
    Then there exists $t_1 \in [-1,0]$ such that $m_0(t_1)=0$. 
    
    We now claim that $m_0(t_0)>0$ for some $t_0 \in [-2,-1]$. Suppose by way of contradiction that $m_0 \equiv 0$ in $[-2,-1]$. 
    From Corollary \ref{c:divint}, we have that
     \[
      \begin{aligned}
      C \Lambda \sup_{[-2,-1]} m_k &\geq 
      \frac{\Lambda}{\Gamma(1-\alpha_k)} \int_{-2}^{-1} \frac{m_k(t)}{(-1-t)^{\alpha}} \ dt \\
         &\geq \int_{-2}^{-1} c_0 f_k(t)- C_1 m_k \ dt\\
         &\geq c_0 \mu - \int_{-2}^{-1}C_1 m_k \ dt.
        \end{aligned}
     \] 
    Letting $k \to \infty$ we obtain that $0 \geq \mu$. Which is a contradiction, and therefore the claim that $m_0$ is not identically zero is true. 
    
     We now consider two different cases. First assume that for a subsequence $\alpha_k \to \alpha_1 <1$. 
     Let $t_2$ be the first point after $t_0$ such that $m_0(t_2)=0$. Because $m_k \to m_0$ uniformly, 
     and since $m_0(t_0)>0$, we may choose $\psi \geq 0$ smooth with $\psi(t)=0$ in a neighborhood of $t_2$ and $\psi(t_0)>0$ and also 
     satisfying $m_k \geq \psi$ for all $k$. Now $M_{\alpha}^+ \psi(t) \leq -\delta_1$ for $t \in [t_2 - \delta_2]$ and 
     $\alpha \in (\alpha_1 - \delta_3, \alpha + \delta_3)$. 
     We let $\e_k$ be such that $\psi + \e_k \leq m_k$ on $[t_2 -\delta_2,t_2]$ and $\psi(t_k) + \e_k = m_k(t_k)$
     for some $t_k \in [t_2 - \delta_2, t_2]$. We define  
      \[
       \phi_k := \begin{cases}
                       \psi + \e_k &\text{ if } t \geq t_2 - \delta_2 \\
                       m_k &\text{ if } t < t_2 - \delta_2. 
                      \end{cases}
      \] 
     Then
      \[
       M_{\alpha}^{+} \phi_k(t_k) \geq \D_t^{\alpha_k} \phi_k (t_k) \geq -C_1 m_k(t_k).
      \] 
    Since $m_0(t)>0$ for $t\in (t_2 - \delta_2, t_2)$, as $k \to \infty$, we have  $t_k \to t_2$ and $\e \to 0$. 
    Then 
      \[
        \lim_{k \to \infty} M_{\alpha_k}^{+} \phi_k(t_k) \geq \lim_{k \to \infty} - C_1 u_{k}(t_k) =0. 
      \]
    But we also have  
     \[
      \lim_{k \to \infty} M_{\alpha_k}^{+} \phi_k(t_k) \leq  
       M_{\alpha_1}^+ \psi(t_2)  \leq - \delta_1 <0.
     \]
    This is a contradiction to the first case. 
    
    We now consider the case in which $\alpha_k \to 1$. For a further subsequence there exists $f_0$
    such that $f_k$ converges to $f_0$ in weak star $L^{\infty}$, so that 
     \[
     \int_{-2}^{-1} f_0(t) \geq \mu. 
   \]    
   Then from Corollary \ref{c:divint} we have that
    \[
     \frac{\lambda\alpha_k}{\Gamma(1-\alpha_k)} \int_{-2}^{t} m_k(s) \ ds \leq 
     \int_{-2}^t f_k(s) - C_1m_k(s) \ ds 
     \leq  \frac{\Lambda\alpha_k}{\Gamma(1-\alpha_k)} \int_{-2}^{t} m_k(s) \ ds.    
    \]
    Since $m_k \to m_0$ uniformly and since \cite{d04} for any continuous function $h$ 
    we have
    \[
    \frac{\alpha_k}{\Gamma(1-\alpha_k)} \int_{-2}^t \frac{h(s)}{(t-s)^{\alpha_k}} \ ds \to h(t), 
    \]
    as $\alpha_k \to 1$,
    then we obtain as $k \to \infty$ the inequality 
     \[
      \lambda m_0(t)  \leq \int_{-2}^t f_0(s) - C_1 m_0(s)  \ ds \leq \Lambda m_0(t). 
     \]
    Then there exists $\lambda \leq g(t) \leq \Lambda $ such that 
     \[
      g(t)m_0(t) = \int_{-2}^t f_0(s) - C_1 m_0(s)  \ ds.
     \]
     Then $gm_0$ is Lipschitz continuous since $f_0 $ is bounded. Furthermore, we have that 
      \[
       g(t)m_0(t) \geq \int_{-2}^t f_0(s) - C_1 \lambda^{-1}g(s)m_0(s)  \ ds.
      \]
     Then $gm_0$ is a supersolution and from the theory of ordinary differential equations, $gm_0 \geq \tilde{g}$ with 
     $\tilde{g}$ solving
      \[
       \tilde{g}(t) = \int_{-2}^t f_0(s) - C_1 \lambda^{-1}\tilde{g}(s)  \ ds.
      \]
     Since $f_0$ is not identically zero, then $\tilde{g}>0$ on $[-1,0]$.  It follows that $gm_0>0$ and hence
     $m_0>0$ on $[-1,0]$ as well. This is a contradiction to the second case. 
     Then for fixed $c_0,C_1, \mu$, there exists $\delta_4>0$ depending on $c_0,C_1, \mu, \lambda, \Lambda, \alpha_0$
     such that for any solution satisfying the assumptions in the statement of the 
     lemma, we have that $u \geq \delta_4$  in $[-1,0]$. We then obtain a modulus of continuity as stated in the Lemma. 
   \end{proof}

\section{H\"older Continuity}  \label{s:holder2} 
  
 In this section we follow the method used in \cite{s11} to prove our main result. 
 We will need the following Lemma to account for the growth in the tails.  
    
  \begin{lemma}  \label{l:decrease}
   Let $u$ be a continuous function, $u\leq 1$ in $(\R^n \times [-2,0])\cup (B_2 \times [-\infty,0])$, which satisfies 
   the following inequality in the viscosity sense in $B_2 \times [-2,0]$
    \begin{equation}  \label{e:assume1}
     \D_t^{\alpha} u - M^+ u \leq \epsilon_0.
    \end{equation}
   with $\alpha_0 \leq \alpha <1$. 
   Assume also that  
    \begin{equation}  \label{e:assume2}
     | \{u\leq 0\} \cap (B_1 \times [-2,-1])| \geq \mu.
    \end{equation}
   Then if $\epsilon_0$ is small enough there exists $\theta>0$ such that $u \leq 1-\theta$ in $B_1 \times [-1,0]$. 
   The maximum value of $\epsilon_0$ as well as $\theta$ depend only on $\alpha_0, \lambda, \Lambda,n$ and $\sigma$.
  \end{lemma}

  \begin{proof}
   We first mention that it is sufficient to prove the Lemma under the assumption 
    \[
     | \{u< 0\} \cap (B_1 \times [-2,-1])| \geq \mu.
    \]
   For if $u$ satisfies \eqref{e:assume1} and \eqref{e:assume2}, then $u-c$ for any positive constant $c$ will satisfy \eqref{e:assume1} as well as the inequality above.
   Then $u-c \leq 1 -\theta$ in $B_1 \times [-1,0]$ independent of $c$ and so letting $c \to 0$ one obtains $u\leq 1 -\theta$. 
    
   We consider the fractional ordinary differential equation $m: (-\infty,0] \to \R$
    \begin{equation}  \label{e:fode}
     \begin{aligned}
      m(t)&=0 \quad &\text{ for } t\leq -2\\
      \D_t^{\alpha}  m(t) &= c_0 f(t) - C_1m(t) \quad &\text{ for } t>-2,
     \end{aligned}
    \end{equation}
   Where $f(t)$ is a smooth approximation from below of $|\{x \in B_1: u(x,t)<0\}|$. 
   From the hypothesis and Lemma \ref{l:mubound}, we can choose an approximation $f$ such that 
    \[
     m(t) \geq \omega_1(c_0\mu) \omega_2(C_1^{-1})/2 >0,
    \]
   for $t \in [-1,0]$. 
   By Lemma \ref{l:zorn} we can approximate $m$ uniformly from below by a Lipschitz function $g$ such that 
    \[
     \begin{aligned}
      g(-2)&=0 \\
      \D_t^{\alpha}  g(t) &\leq c_0 f(t) - C_1m(t),
     \end{aligned}    
    \]
   and 
    \[
     g(t) \geq \omega_1(c_0\mu) \omega_2(C_1^{-1})/4 >0,
    \]   
    for $t \in [-1,0]$. We utilize the function $g(t)$ which is not just a viscosity solution but also a classical solution 
    since it is Lipschitz. We can then calculate $\D_t^{\alpha} g$ everywhere classically. Furthermore, $g$ is allowed as a test function for 
    touching above or below for viscosity solutions.  
    
    We want to show that $u \leq 1- g(t) + \epsilon_0 c_{\alpha}2^{\alpha}$ if $c_0$ is small and $C_1$ is large. We can then set 
    $\theta = \omega_1(c_0\mu) \omega_2(C_1^{-1})/4$ for $\epsilon_0$ small to obtain the result of the Lemma. 
   We pick the constant $c_{\alpha}$ such that $\partial_t^{\alpha} c_{\alpha}(2+t)_+^{\alpha} = 1$ for $t>-2$, and note \cite{d04}
   that $c_{\alpha}$ is uniform as $\alpha \to 1$. 
   Let $\beta:\R \to \R$ be a fixed smooth nonincreasing function such that $\beta(x)=1$ if $x\leq 1$ and $\beta(x)=0$ if $x\geq 2$. 
   Let $b(x)= \beta(|x|)$. Where $b=0$ we have $M^- b>0$. Since $b$ is smooth $M^-b$ is continuous and it remains positive for $b$ small enough (\cite{s11}).
   Thus there exists $\beta_1$ such that $M^- b \geq 0$ if $b \leq \beta_1$.
   
   Assume that there exists some point $(x,t) \in B_1 \times [-1,0]$ such that 
    \[
     u(x,t) > 1 -g(t)+\epsilon_0 \lambda^{-1} c_{\alpha} (2+t)_+^{\alpha}. 
    \]
   We will arrive at a contradiction
   by looking at the maximum of the function
    \[
     w(x,t) =u(x,t) +g(t)b(x) - \epsilon_0 \lambda^{-1} c_{\alpha}(2+t)_+^{\alpha}
    \]
   We assume there exists a point in $B_1 \times [-1,0]$ where $w(x,t)>1$. Let $(x_0,t_0)$ be the point that realizes the maximum of $w$:
    \[
     w(x_0,t_0) = \max_{\R^n \times (-\infty,0]}w(x,t).
    \]  
   This maximum is larger than 1, and so it must be achieved when $t> -2$ and $|x|<2$. 
   
   Let $\phi(x,t):= w(x_0,t_0)-g(t)b(x)+\e_0 \lambda^{-1} c_{\alpha}(2+t)_+^{\alpha}$. We remark that $\phi(x,t)=\phi(x,-2)$ for 
   $t\leq -2$, and $\phi$ touches $u$ from above at the point $(x_0,t_0)$. 
   We define
    \[
     v(x,t) := 
              \begin{cases}
                 \phi(x,t) & \text{if } x \in B_r  \\ 
                    u(x,t) & \text{if } x \notin B_r.\\
    \end{cases}
    \] 
   Then from the definition of viscosity solution we have
    \begin{equation}  \label{e:viscos}
     \partial_t^{\alpha}v - M^+v \leq \epsilon_0 \quad \text{ at } \quad (x_0,t_0). 
    \end{equation}
   We have that 
    \[
     \begin{aligned}
         \D_t^{\alpha} v(x_0,t_0) &= \D_t^{\alpha} \left( -g(t_0)b(x_0) +\epsilon_0\lambda^{-1}c_{\alpha}(2+t)_+^{\alpha} \right)\\
                                            &\geq (C_1 m(t) - c_0 f(t))b(x_0) + \e_0 \lambda^{-1} M_{\alpha}^- (2+t)_+^{\alpha} \\
                                            &= (C_1 m(t) - c_0 f(t))b(x_0) + \e_0 \lambda^{-1} \lambda\partial_t^{\alpha} (2+t)_+^{\alpha} \\
                                            &=(C_1 m(t) - c_0f(t))b(x_0) + \e_0 \\                                            
                                            &\geq(C_1 m(t) - c_0|\{x \in B_1: u(x,t_0)<0\}|)b(x_0) + \e_0  \\
                                            &\geq (C_1 m(t) - c_0|\{x \in B_1: u(x,t_0)\leq 0\}|)b(x_0) + \e_0. \\                                            
     \end{aligned}
     \] 
   Then 
    \[
     \begin{aligned}
     \e_0 &\geq \D_t^{\alpha} v(x_0,t_0) - M^+v(x_0,t_0) \\
             &\geq (C_1 m(t) - c_0|\{x \in B_1: u(x,t_0)\leq 0\}|)b(x_0) + \e_0 - M^+v(x_0,t_0),
     \end{aligned}
    \]
    or 
    \begin{equation}  \label{e:timepart1}
     0 \geq (C_1 m(t) - c_0|\{x \in B_1: u(x,t_0)\leq 0\}|)b(x_0)  - M^+v(x_0,t_0).
    \end{equation}
   Now exactly as in \cite{s11} we obtain the following bound for $G := \{x \in B_1 \mid u(x,t_0)\leq 0 \}$,
    \begin{equation}  \label{e:silv}
     M^+ v(x_0,t_0) \leq -m(t_0) M^-b(x_0,t_0) -c_0|G \setminus B_r|
    \end{equation}
   for some universal constant $c_0$. This is how we choose $c_0$ in the fractional ordinary differntial equation. 
   We now look at two different cases and obtain a contradiction in both. Suppose $b(x_0)\leq\beta_1$. Then $M^-b(x_0)\geq0$, and so
   from \eqref{e:silv} 
    \[
     M^+v(x_0,t_0)\leq -c_0|G\setminus B_r|. 
    \]
   Combining the above inequality with \eqref{e:timepart1}, we obtain
    \[
     0 \geq \left(-c_0 |\{x \in B_1: u(x,t)\leq 0\}| + C_1m(t)\right)b(x_0) + c_0|G\setminus B_r|.
    \]
   For any $C_1>0$ this will be a contradiction by taking $r$ small enough. 
   
   Now suppose $b(x_0)>\beta_1$. Since $b$ is a smooth compactly supported function, there exists $C$ such that $|M^-b|\leq C$. We then have 
    from \eqref{e:silv} the bound
    \[
     M^+ v(x_0,t_0) \leq Cm(t_0) -c_0|G \setminus B_r|
    \]  
   and inserting this in \eqref{e:timepart1} we have
    \[
     0 \geq \left(-c_0 |\{x \in B_1: u(x,t)\leq 0\}| + C_1m(t)\right)b(x_0)  - Cm(t_0) +c_0|G \setminus B_r|
    \]
   Letting $r \to 0$ we obtain 
    \[
     \begin{aligned}
     0  &\geq c_0(1-b(x_0))|G| + (C_1b(x_0)-C)m(t_0) \\
        &\geq c_0(1-b(x_0))|G| + (C_1\beta_1-C)m(t_0).
     \end{aligned}
    \]
   Choosing $C_1$ large enough we obtain a contradiction. 
  \end{proof}

 We now define
  \[
   Q_r := B_r \times [-r^{2\sigma/\alpha},0],
  \]
 and note the rescaling property that if $v(x,t)=u(rx,r^{2\sigma/\alpha}t)$, then 
  \[
   \tilde{\D}_t^{\alpha}v(x,t) - M_{\sigma}^{\pm}v(x,t) = 
     r^{2\sigma} \left(\D_t^{\alpha}u(rx,r^{2\sigma/\alpha}t) - M_{\sigma}^{\pm}u(rx,r^{2\sigma/\alpha}t) \right)
  \]
  where if $\K(t,s)$ is the kernel for $\D_t^{\alpha}$,  then $\tilde{\D}_t^{\alpha}$ has kernel 
   \[
    \frac{\K(rt,rs)}{r^{1+\alpha}}
   \]
  which will also satisfy \eqref{e:timediv} and \eqref{e:alphabound}. 
  For the next three results we fix $r=\min\{4^{-1},4^{-\alpha/2\sigma}\}$.
  We will need the following Proposition to bound the tails. 
  \begin{proposition}  \label{p:nubound}
   Let $h(t)= \max{\{2|rt|^{\nu}-1,0\}}$ with $r= \min\{4^{-1},4^{-\alpha/2\sigma}\}$. 
   If $t_1\leq 0$ and $\nu <\alpha$ then 
    \[
     0 \geq \D_t^{\alpha} h(t_1) \geq - \lambda c_{\alpha, \nu}
    \]
   where $c_{\alpha,\nu}$ is a constant depending only on $\alpha$ and $\nu$ but for fixed $\nu$ remains uniform as $\alpha \to 1$.
  \end{proposition}
 
  \begin{proof}
   Now 
    \[
     0 \geq \D_t^{\alpha} h(t_1) \geq M_{\alpha}^- h(t_1) = \lambda \partial_t^{\alpha} h(t_1). 
    \]
   From \cite{a17} we have
    \[
     \partial_t^{\alpha} h(t_1) \geq c_{\alpha, \nu},  
    \]
   which for fixed $\nu$ remains uniform as $\alpha \to 1$. Combining the above two inequalities, the Proposition is proven. 
  \end{proof}

 \begin{lemma}  \label{l:down}
  Let $u$ be a bounded continuous function which satisfies the following two inequalities in the viscosity sense in $Q_1$
   \begin{equation}   \label{e:pucci}
    \begin{aligned}
     &\D_t^{\alpha} u - M^+ u \leq \epsilon_0/2, \\
     &\D_t^{\alpha} u - M^- u \geq -\epsilon_0/2.
    \end{aligned}
   \end{equation}
  Let the kernel $\K(t,s)$ of $\D_t^{\alpha}$ satisfy \eqref{e:timediv} and \eqref{e:alphabound} with $0<\alpha_0\leq \alpha <1$.  
  Then there are univeral constants $\theta>0$ and $\nu>0$ depending only on $n,\sigma,\Lambda,\lambda,\alpha_0$
  such that if 
   \[
    \begin{aligned}
     |u| \leq 1                     &\quad \text{in} \quad B_1 \times [-1,0] \\
     |u(x,t)| \leq 2|rx|^{\nu}-1 &\quad \text{in} \quad (\R^n\setminus B_1) \times [-1,0] \\
     |u(x,t)| \leq 2|rt|^{\nu}-1 &\quad \text{in} \quad B_1 \times (-\infty,-1], 
    \end{aligned}   
   \]
  with $r=\min\{4^{-1},4^{-\alpha/2\sigma}\}$, then 
   \[
    \text{osc}_{Q_{r}} u \leq 1-\theta
   \]
 \end{lemma}

 \begin{proof}
  We consider the rescaled version 
   \[
    \tilde{u}(x,t) := u(r^{-1}x,r^{-2\sigma/\alpha} t). 
   \]
  The function $\tilde{u}$ will stay either positive or negative in half of the points in $B_1 \times [-2,-1]$. Let us assume that 
   $\{\tilde{u} \leq 0\}\cap (B_1 \times [-2,-1])\geq |B_1|/2$. Otherwise we can repeat the proof for $-\tilde{u}$. We would like to apply Lemma \ref{l:decrease}.
   To do so we would need $\tilde{u}\leq 1$. We consider $v:= \min\{1,\tilde{u}\}$. Inside $Q_{r^{-1}}$ we have $v=\tilde{u}$. The error comes only from the tails in 
   the computations. Exactly as in \cite{s11} we obtain for $\kappa$ small enough
    \[
     -M^+ v \leq -M^+ \tilde{u} + \epsilon_0/4.
    \]
   From Proposition \ref{p:nubound} we have for small enough $\kappa$ that
    \[
     \D_t^{\alpha} v \leq \D_t^{\alpha} \tilde{u} + \epsilon_0/4.
    \]
   Thus 
    \[
     \D_t^{\alpha} v - M^+ v \leq \epsilon_0.
    \]
   We now apply Lemma \ref{l:decrease} to $v$ and rescale back to conclude the proof. 
 \end{proof}

We are now able to give the proof of our main result. 
%
%
%
 \begin{proof}[Theorem \ref{t:main2}]
  We first choose $\kappa<\nu$ for $\nu$ as in Lemma \ref{l:down}
  Let $(x_0,t_0) \in Q_1$. We consider the rescaled function 
   \[
    v(x,t) = \frac{u(x_0 + x,t+t_0}{\| u\|_{L^{\infty}} + \epsilon_0^{-1} \| f\|_{L^{\infty}}},
   \]
  and note that $|v|\leq 1$ and $v$ is a solution to 
   \[
    \begin{aligned} 
     \D_t^{\alpha} v - M_{\sigma}^+v &\leq \e_0 \\
     \D_t^{\alpha} v - M_{\sigma}^-v &\geq -\e_0.
    \end{aligned}
   \]
  in $B_2 \times [-1,0]$. We let $r=\min\{4^{-1},4^{-\alpha/2\sigma}\}$ and the estimate will follow as soon as we show
   \begin{equation}  \label{e:osc2}
    \text{osc}_{Q_{r_k}} v \leq 2 r^{\kappa k}.
   \end{equation}
  Estimate \eqref{e:osc2} will be proven by constructing two sequences $a_k \leq v \leq b_k$ in $Q_{r_k}$, $b_k - a_k= 2r^{\kappa k}$ with 
  $a_k$ nondecreasing and $b_k$ nonincreasing. The sequence is constructed inductively. 
  
  Since $|v|\leq 1$ everywhere, we can start by choosing some $a_0 \leq \inf v$ and $b_0 \geq \sup v$ so that $b_0 - a_0=2$. Assuming now that
  the sequences have been constructed up to the value $k$  we scale
   \[
    w(x,t) = (v(r^k x,r^{2k\sigma /\alpha}t) -(a_k+b_k)/2)r^{-\kappa k}.
   \] 
  We then have 
   \[
    \begin{aligned}
     |w| \leq 1                 &\quad \text{in } \quad  Q_1 \\
     |w| \leq 2r^{-\kappa k} -1 &\quad \text{in } \quad  Q_{r^{-k}}.
    \end{aligned}
   \]
  and so
   \[
    \begin{aligned}
     |w(x,t)|\leq 2|x|^{\nu}-1  &\quad \text{for } \quad (x,t) \in B_1^c \times [-1,0] \\
     |w(x,t)|\leq 2|t|^{\nu}-1  &\quad \text{for } \quad (x,t) \in B_1 \times (-\infty,-1).
    \end{aligned}
   \]
  Notice also that $w$ has new right hand side bounded by 
   \[
    \epsilon_0 r^{k(\kappa-2\sigma)}
   \]
  which is strictly smaller than $\epsilon_0$ for $\kappa<2\sigma$. 
   For $\kappa$ small enough we can apply Lemma \ref{l:down} to obtain 
   \[
    \text{osc}_{ Q_r}  w \leq 1-\theta. 
   \]
  Then if $\kappa$ is chosen smaller than the $\kappa$
  in Lemma \ref{l:down} and also so that $1-\theta \leq r^{\kappa}$, then this implies
   \[
    \text{osc}_{ Q_{r^{k+1}}} w \leq r^{\kappa(k+1)}
   \]
  so we can find $a_{k+1}$ and $b_{k+1}$ and this finishes the proof.  
 \end{proof}

\bibliographystyle{amsplain}
\bibliography{refmarchaud}

\end{document}